\documentclass[11pt,english]{article}
\usepackage[latin9]{inputenc}
\usepackage{amsmath}
\usepackage{amsthm}
\usepackage{amssymb}
\usepackage[numbers]{natbib}

\makeatletter
\theoremstyle{definition}
\newtheorem{example}{\protect\examplename}[section]
\theoremstyle{plain}
\newtheorem{prop}{\protect\propositionname}[section]
\theoremstyle{plain}
\newtheorem{cor}{\protect\corollaryname}[section]
\theoremstyle{remark}
\newtheorem{rem}{\protect\remarkname}[section]
\theoremstyle{plain}
\newtheorem{thm}{\protect\theoremname}[section]
\theoremstyle{plain}
\newtheorem{lem}{\protect\lemmaname}[section]

\usepackage[all,2cell]{xy}  \UseAllTwocells \SilentMatrices
\usepackage{hyperref}
\usepackage{amsmath}
\usepackage{slashed}
\usepackage{graphicx}

\usepackage[misc,geometry]{ifsym} 

\hypersetup{colorlinks=true, linkcolor=blue, citecolor=blue, filecolor=blue, pagecolor=blue, urlcolor=red, pdfauthor={}, pdftitle={B_k_manifolds}}

\makeatother

\usepackage{babel}
\providecommand{\corollaryname}{Corollary}
\providecommand{\examplename}{Example}
\providecommand{\lemmaname}{Lemma}
\providecommand{\propositionname}{Proposition}
\providecommand{\remarkname}{Remark}
\providecommand{\theoremname}{Theorem}

\begin{document}
\title{Existence of $B_{\alpha,\beta}^{k}$-Structures on $C^{k}$-Manifolds}
\author{Yuri Ximenes Martins\footnote{yurixm@ufmg.br (corresponding author)}\,  and  Rodney Josu\'e Biezuner\footnote{rodneyjb@ufmg.br} }
\maketitle
\noindent \begin{center}
\textit{Departamento de Matem\'atica, ICEx, Universidade Federal de Minas Gerais,}  \\  \textit{Av. Ant\^onio Carlos 6627, Pampulha, CP 702, CEP 31270-901, Belo Horizonte, MG, Brazil}
\par\end{center}
\begin{abstract}
In this paper we introduce $B_{\alpha,\beta}^{k}$-manifolds as generalizations
of the notion of smooth manifolds with $G$-structure or with $k$-bounded
geometry. These are $C^{k}$-manifolds whose transition functions
$\varphi_{ji}=\varphi_{j}\circ\varphi_{i}^{-1}$ are such that $\partial^{\mu}\varphi_{ji}\in B_{\alpha(r)}\cap C^{k-\beta(r)}$
for every $\vert\mu\vert=r$, where $B=(B_{r})_{r\in\Gamma}$ is some
sequence of presheaves of Fr\'echet spaces endowed with further structures,
$\Gamma\subset\mathbb{Z}_{\geq0}$ is some parameter set and $\alpha,\beta$
are functions. We present embedding theorems for the presheaf category
of those structural presheaves $B$. The existence problem of $B_{\alpha,\beta}^{k}$-structures
on $C^{k}$-manifolds is studied and it is proved that under certain
conditions on $B$, $\alpha$ and $\beta$, the forgetful functor
from $C^{k}$-manifolds to $B_{\alpha,\beta}^{k}$-manifolds has adjoints.
\end{abstract}

\section{Introduction}

$\quad\;\,$One can think of a ``$n$-dimensional manifold'' as
a topological space $M$ in which we assign to each neighborhood $U_{i}$
in $M$ a bunch of coordinate systems $\varphi_{i}:U_{i}\rightarrow\mathbb{R}^{n}$.
The regularity of $M$ is determined by the space in which the transition
functions $\varphi_{ji}:\varphi_{i}(U_{ij})\rightarrow\mathbb{R}^{n}$
live, where $U_{ij}=U_{i}\cap U_{j}$ and $\varphi_{ji}=\varphi_{j}\circ\varphi_{i}^{-1}$.
Indeed, if $B:\operatorname{Open}(\mathbb{R}^{n})^{op}\rightarrow\mathbf{Fre}$
is some presheaf of Fr\'echet spaces in $\mathbb{R}^{n}$, we can
define a \emph{$n$-dimensional $B$-manifold} as one whose regularity
is governed by $B$, i.e, whose transition functions $\varphi_{ji}$
belong to $B(\varphi_{i}(U_{ij}))$. For instance, given $0\leq k\leq\infty$,
a $C^{k}$-manifold (in the classical sense) is just a $C^{k}$-manifold
in this new sense if we regard $C^{k}$ as the presheaf of $k$-times
continuouly differentiable functions.

We could study $B$-manifolds when $B\subset C^{k}$, calling them
$B^{k}$\emph{-manifolds}.\emph{ }For instance, if $k\geq1$, then
any $C^{k}$-manifold can always be regarded as a $C^{\infty}$-manifold
\citep{whitney_embedding}, showing that smooth manifolds belong to
that class of $B$-manifolds, since $C^{\infty}\subset C^{k}$. Another
example are the analytic manifolds, for which $B=C^{\omega}$. Recall
that for $B=C^{k}$, with $k>1$, we have regularity conditions not
only on the transition functions $\varphi_{ji}$, but also on the
derivatives $\partial^{\mu}\varphi_{ji}$, for each $\vert\mu\vert\leq k$.
In fact, $\partial^{\mu}\varphi_{ji}\in C^{k-\vert\mu\vert}(\varphi_{i}(U_{ij}))$.
We notice, however, that if $B\subset C^{k}$ is a subpresheaf, then
$B\subset C^{k-\vert\mu\vert}$, but it is not necessarily true that
$\partial^{\mu}\varphi_{ji}\in B(\varphi_{i}(U_{ij}))$. For instance,
if $B=L^{p}\cap C^{\infty}$, then $\varphi_{ji}$ are smooth $L^{p}$-integrable
functions, but besides being smooth, the derivatives of a $L^{p}$-integrable
smooth functions need not be $L^{p}$-integrable \citep{key-1}.

The discussion above motivates us to consider $n$-dimensional manifolds
which satisfy \emph{a priori} regularity conditions on the transition
functions and also on their derivatives. Indeed, let now $B=(B_{r})_{r\in\Gamma}$
be a sequence of presheaves of Fr\'echet spaces with $r\in\mathbb{Z}_{\geq0}$
and redefine a \emph{$B^{k}$-manifold }as one in which $\partial^{\mu}\varphi_{ji}\in B_{r}(\varphi_{i}(U_{ij}))\cap C^{k-r}(\varphi_{i}(U_{ij}))$
for every $r\leq k$ and $\vert\mu\vert=r$. Many geometric objects
can be put in this new framework. Just to exemplify we mention two
of them.
\begin{example}
\label{example_1_int} Given a linear group $G\subset GL(n;\mathbb{R})$,
if we take $B_{0}=C^{k}$, $B_{1}$ as the presheaf of $C^{1}$-functions
whose jacobian matrix belongs to $G$ and $B_{r}=C^{k-r}$, with $r\geq2$,
then a $B^{k}$-manifold describes a $C^{k}$-manifold endowed with
a $G$-structure in the sense of \citep{kobayashi,sternberg}. This
example includes many interesting situations, such as semi-riemannian
manifolds, almost-complex manifolds, regularly foliated manifolds,
etc.
\begin{example}
\label{example_2_int} For a fixed $p\in[1,\infty]$, consider $B_{r}=L^{p}$
for every $r$. Then a $B^{k}$-manifold is a $C^{k}$-manifold whose
transition functions and its derivatives are $L^{p}$-integrable.
This means that $\varphi_{ji}\in W^{k,p}(\varphi_{i}(U_{ij}))\cap C^{k}(\varphi_{i}(U_{ij}))$,
where $W^{k,p}(\Omega)$ are the Sobolev spaces, so that a $(L^{p})^{k}$-manifold
is a ``Sobolev manifold''. The case $p=\infty$ corresponds to the
notion of $k$-bounded structure on a $n$-manifold \citep{bounded_geometry_1,bounded_geometry_2,bounded_geomery_3}.
\end{example}
\end{example}
This is the first of a sequence of articles which aim to begin the
development of a general theory of $B^{k}$-manifolds, including the
unification of certain aspects of $G$-structures and Sobolev structures
on $C^{k}$-manifolds. Our focus is on regularity results for geometric
objects on $B^{k}$-manifolds. The motivation is as follows. When
studying geometry on a smooth manifold $M$ we usually know that geometric
objects $\tau$ on $M$ (e.g, affine connections, bundles, nonlinear
differential operators of a specific type, solutions of differential
equations, etc.) exist with certain regularity, meaning that their
local coefficients $\tau_{j_{1},...,j_{s}}^{i_{1},...,i_{r}}$ belong
to some space, but in many situations we would like to have existence
of more regular objects, meaning to have $\tau$ such that $\tau_{j_{1},...,j_{s}}^{i_{1},...,i_{r}}$
belong to a smaller space.

For instance, in gauge theory and when we need certain uniform bounds
on the curvature, it is desirable to work with connections whose coefficients
are not only smooth, but actually $L^{p}$-integrable or even uniformly
bounded \citep{gauge,bounded_geometry_1}. In the study of holomorphic
geometry, we consider holomorphic or hermitian connections over complex
manifolds, whose coefficients are holomorphic \citep{holomorphic_1,atiyah_class}.
It is also interesting to consider symplectic connections, which are
important in formal deformation quantization of symplectic manifolds
\citep{symplectic_connection,deformation_quantization}. These examples
immediately extend to the context of differential operators, since
they can be regarded as functions depending on connections on vector
bundles over $M$ \citep{heat_kernel}. The regularity problem for
solutions of differential equations is also very standard and by means
of Sobolev-type embedding results and bootstrapping methods it is
strongly related to the regularity problem for the underlying differential
operator defining the differential equation \citep{key-1}.

Since the geometric structures are defined on $M$, the existence
of some $\tau$ such that $\tau_{j_{1},...,j_{s}}^{i_{1},...,i_{r}}$
satisfy a specific regularity condition should depend crucially on
the existence of a (non-necessarily maximal) subatlas for $M$ whose
transition functions themselves satisfy additional regularity conditions,
leading us to consider $B^{k}$-manifolds in the sense introduced
above. Let $S$ be a presheaf of Fr\'echet spaces and let us say
that $\tau$ is \emph{$S$-regular} if there exists an open covering
of $M$ by charts $\varphi_{i}:U_{i}\rightarrow\mathbb{R}^{n}$ in
which $\tau_{j_{1},...,j_{s}}^{i_{1},...,i_{r}}\in S(\varphi_{i}(U_{ij}))$.
The geometric motivation above can then be rephrased into the following
problem:
\begin{itemize}
\item let $M$ be a $C^{k}$-manifold. Given $S$, find some presheaf $B(S)$,
depending on $S$, such that if $M$ has a $B(S)^{k}$-structure,
then there exist $S$-regular geometric structures on it.
\end{itemize}
$\quad\;\,$In this sequence of articles we intend to work on this
question for different kinds of geometric structures: affine connections,
nonlinear tensors, differential operators and fiber bundles. Notice,
on the other hand, that if $S$ is a presheaf for which the previous
problem has a solution for some kind of $S$-regular geometric structures,
we have an immediate existence question:
\begin{itemize}
\item under which conditions does a $C^{k}$-manifold admit a $B(S)^{k}$-strucure?
\end{itemize}
$\quad\;\,$Opening this sequence of works, in the present one we
will discuss the general existence problem of $B^{k}$-structures
on $C^{k}$-manifolds. Actually, we will work on more general entities,
which we call $B_{\alpha,\beta}^{k}$-\emph{manifolds}, where $\alpha$
and $\beta$ are functions and the transition functions satisfy $\partial^{\mu}\varphi_{ji}\in B_{\alpha(r)}(\varphi_{i}(U_{ij}))\cap C^{k-\beta(r)}(\varphi_{i}(U_{ij}))$.
Thus, for $\alpha(r)=r=\beta(r)$ a $B_{\alpha,\beta}^{k}$-manifold
is the same as a $B^{k}$-manifold. We also study some aspects of
the presheaf-category $\mathbf{C}_{n,\beta}^{k,\alpha}$ of the presheaves
$B$. The present paper has the following two main results:$\underset{\underset{\;}{\;}}{\;}$

\noindent \textbf{Theorem A.} \emph{There are full embeddings}
\begin{enumerate}
\item[(1)] \emph{$\imath:\mathbf{C}_{n,\beta;l}^{k,\alpha}\hookrightarrow\mathbf{C}_{n,\beta}^{l,\alpha}$,
if $l\leq k$;}
\item[(2)] \emph{$f_{*}:\mathbf{C}_{n,\beta}^{k,\alpha}\hookrightarrow\mathbf{C}_{r,\beta}^{k,\alpha}$,
for any continuous injective map $f:\mathbb{R}^{n}\rightarrow\mathbb{R}^{r}$;}
\item[(3)] \emph{$\jmath:\mathbf{C}_{n,\beta;\beta'}^{k,\alpha}\hookrightarrow\mathbf{C}_{r,\beta'}^{k,\alpha}$,
if $\beta'\leq\beta$.}$\underset{\underset{\;}{\;}}{\;}$
\end{enumerate}
\noindent \textbf{Theorem B.} \emph{If $B$ is ordered, fully left-absorbing
(resp. fully right-absorbing) and has retractible $(B,k,\alpha,\beta)$-diffeomorphisms,
all of this in the same intersection presheaf $\mathbb{X}$, then
the choice of a retraction $\overline{r}$ induces a left-adjoint
(resp. right-adjoint) for the forgetful functor $F$ from $B_{\alpha,\beta}^{k}$-manifolds
to $C^{k}$-manifolds, which actually independs of $\overline{r}$.
In particular, if $B$ is fully absorbing, then $F$ is ambidextrous
adjoint.$\underset{\underset{\;}{\;}}{\;}$}

This article is structured as follows. In Section \ref{sec_Gamma_spaces}
we introduce the classes of Fr\'echet spaces which will be used in
the next sections. These are the $(\Gamma,\epsilon)$-\emph{spaces},
where $\Gamma$ is a set of indexes (in general $\Gamma\subset\mathbb{Z}_{\geq0}$)
and $\epsilon:\Gamma\times\Gamma\rightarrow\Gamma$ is some function.
The $(\Gamma,\epsilon)$-spaces are itself sequences $B=(B_{i})_{i\in\Gamma}$
of nuclear Fr\'echet spaces. The map $\epsilon$ is used in order
to consider \emph{multiplicative structures} on $B$, which are given
by a family of continuous linear maps $*_{ij}:B_{i}\otimes B_{j}\rightarrow B_{\epsilon(i,j)}$,
where the tensor product is the projective one. Many other structures
on $B$ are required, such as additive structures, distributive structures,
intersection structures and closure structures.

In Section \ref{sec_C_k_a_b} we define the $C_{\alpha,\beta}^{k}$-presheaves,
which are the structural presheaves for the $B_{\alpha,\beta}^{k}$-manifolds,
and study the presheaf category of them. We begin by constructing
a sheaf-theoretic version of the concepts and results of Section \ref{sec_Gamma_spaces}.
Thus, in this section $B$ is not a single \emph{$(\Gamma,\epsilon)$}-space,
but a presheaf of them on $\mathbb{R}^{n}$. The $C_{\alpha,\beta}^{k}$-presheaves
are those presheaves of \emph{$(\Gamma,\epsilon)$}-spaces which are
well-related with the presheaf $C^{k-\beta}$ in a sense defined in
that section. In Section \ref{sec_thm_A} Theorem A is proved.

In Section \ref{sec_B_k,a,b_manifolds} $B_{\alpha,\beta}^{k}$-manifolds
are more precisely defined and some basic properties are proven. For
instance, we establish conditions on $B$, $\alpha$ and $\beta$,
under which the category $\mathbf{Diff}_{\alpha,\beta}^{B,k}$ of
$B_{\alpha,\beta}^{k}$-manifolds and $B_{\alpha,\beta}^{k}$-morphisms
between them becomes well-defined. Some examples are also given. Finally,
in Section \ref{sec_existence} we introduce the notions of ordered
presheaf, fully left/right-absorbing presheaf and presheaf with retractible
$(B,k,\alpha,\beta)$-diffeomorphisms, which are needed for the first
theorem. A proof of the Theorem B is given and as a consequence we
get the existence of some limits and colimits in the category of $B_{\alpha,\beta}^{k}$-manifolds.

\section{$(\Gamma,\epsilon)$-Spaces \label{sec_Gamma_spaces}}

$\quad\;\,$Let $\mathbf{Set}$ be the category of all sets and let
$\Gamma\in\mathbf{Set}$ be a set of indexes\footnote{In all the article and in the examples considered $\Gamma$ will be
$\mathbb{Z}_{\geq0}$\emph{ }or some finite product of copies of $\mathbb{Z}_{\geq0}$.
However, all the results of this section can be generalized to the
case in which $\Gamma$ is an arbitrary (i.e, not necessarily discrete)
category.}. A \emph{nuclear Fr\'echet} $\Gamma$-\emph{space }(or \emph{$\Gamma$-space},
for short) is a family $(B_{i})_{i\in\Gamma}$ of nuclear Fr\'echet
spaces. Equivalently, it is a $\Gamma$-graded vector space $B=\oplus_{i}B_{i}$
whose components are nuclear Fr\'echet $\Gamma$-spaces. In other
words, it is a functor $B:\Gamma\rightarrow\mathbf{NFre}$ from the
discrete category defined by $\Gamma$ to the category of nuclear
Fr\'echet spaces and continuous linear maps. Morphisms are pairs
$(\xi,\mu)$, where $\mu:\Gamma\rightarrow\Gamma$ is a function and
$\xi$ is a family of continuous linear maps $\xi_{i}:B_{i}\rightarrow B'_{\mu(i)}$,
with $i\in\Gamma$. Thus, it is an endofunctor $\mu$ together with
a natural transformation $\xi:B\Rightarrow B'\circ\mu$. Composition
is defined by horizontal composition of natural transformations.

Let $\mathbf{Fre}_{\Gamma}$ be the category of $\Gamma$-spaces and
notice that we have an inclusion $\imath:[\Gamma;\mathbf{NFre}]\hookrightarrow\mathbf{NFre}_{\Gamma}$
given by $\imath(B)=B$ and $\imath(\xi)=(\xi,id_{\Gamma})$, where
$[\mathbf{C};\mathbf{D}]$ denotes the functor category. Since $\imath$
is faithful it reflects monomorphisms and epimorphisms, which means
that if a morphism $(\xi,\mu)$ in $\mathbf{NFre}_{\Gamma}$ is such
that each $\xi_{i}$ is a monomorphism (resp. epimorphism) in $\mathbf{NFre}$,
then it is a monomorphism (resp. epimorphism) in $\mathbf{NFre}_{\Gamma}$.
We have a functor $\mathcal{F}:\mathbf{Set}^{op}\rightarrow\mathbf{Cat}$
such that $\mathcal{F}(\Gamma)=\mathbf{Fre}_{\Gamma}$.

More generally, let us define a\emph{ $(\Gamma,\epsilon)$-space}
as a $\Gamma$-space $B$ endowed with a map $\epsilon:\Gamma\times\Gamma\rightarrow\Gamma$.
A morphism between a $(\Gamma,\epsilon)$-space $B$ and a $(\Gamma,\epsilon')$-space
$B'$ is a morphism $(\xi,\mu)$ of $\Gamma$-spaces such that the
first diagram below commutes. Let $\mathbf{Fre}_{\Gamma ,\epsilon}$
be the category of $(\Gamma,\epsilon)$-spaces for a fixed $\epsilon$
and let $\mathbf{NFre}_{\Gamma,\Sigma}$ be the category of $(\Gamma,\epsilon)$-spaces
for all $\epsilon$. Notice that $\mathbf{NFre}_{\Gamma,\Sigma}\simeq\coprod_{\epsilon}\mathbf{NFre}_{\Gamma,\epsilon}$
and that there exists a forgetful functor $U:\mathbf{NFre}_{\Gamma,\Sigma}\rightarrow\mathbf{NFre}_{\Gamma}$
that forgets $\epsilon$.
\begin{center}
$
\xymatrix{\ar[d]_{\mu \times \mu} \Gamma  \times \Gamma \ar[r]^{\epsilon} & \Gamma \ar[d]^{\mu} \\
\Gamma \times \Gamma \ar[r]_-{\epsilon '} & \Gamma} \quad
$
$
\xymatrix{ \ar@/_{1cm}/[rrr]_{F_{\Sigma}} \mathbf{NFre}_{\Gamma,\Sigma} \ar[r]^-{\simeq} & \coprod _{\epsilon} \mathbf{NFre}_{\Gamma,\epsilon} \ar[r]^-{\coprod _{\epsilon} F_{\epsilon}} & \coprod _{\epsilon} \mathbf{NFre}_{\Gamma\times\Gamma} \ar[r]^-{\nabla} & \mathbf{NFre}_{\Gamma\times\Gamma}  }
$
\par\end{center}

For any $\epsilon:\Gamma\times\Gamma\rightarrow\Gamma$ we have a
functor $F_{\epsilon}:\mathbf{NFre}_{\Gamma,\epsilon}\rightarrow\mathbf{NFre}_{\Gamma\times\Gamma}$
defined by $F_{\epsilon}(B,\epsilon)=B\circ\epsilon$ and $F_{\epsilon}(\xi,\mu)=(\xi\circ\epsilon,\mu\circ\epsilon)$.
In the following we will write $B_{\epsilon}$ to denote $B\circ\epsilon$.
Notice that $F_{\epsilon}$ is well-defined precisely because of the
commutativity of the first diagram. Thus, we also have a functor $F_{\Sigma}:\mathbf{NFre}_{\Gamma,\Sigma}\rightarrow\mathbf{NFre}_{\Gamma\times\Gamma}$
given by the composition above, where $\nabla$ is the codiagonal.
Furthermore, recalling that the category of Fr\'echet spaces is symmetric
monoidal with the projective tensor product $\otimes$ \citep{treves,costello},
given $\Gamma,\Gamma'$ we have a functor 
\[
\otimes_{\Gamma,\Gamma'}:\mathbf{NFre}_{\Gamma}\times\mathbf{NFre}_{\Gamma'}\rightarrow\mathbf{NFre}_{\Gamma\times\Gamma'}
\]
playing the role of an \emph{external product},\emph{ }given by $(B\otimes B')_{(i,j)}=B_{i}\otimes B_{j}$
on objects and by $(\xi,\mu)\otimes(\eta,\nu)=(\xi\otimes\eta,\mu\times\nu)$
on morphisms. By composing with the diagonal and the forgetful functor
$U$, for each $\Gamma$ we get the functor $\otimes_{\Gamma}$ below.
$$
\xymatrix{\ar@/_{0.7cm}/[rrr]_{\otimes _{\Gamma}} \mathbf{NFre}_{\Gamma,\Sigma} \ar[r]^-{U} & \mathbf{NFre}_{\Gamma} \ar[r]^-{\Delta} & \mathbf{NFre}_{\Gamma} \times \mathbf{NFre}_{\Gamma}  \ar[r]^-{\otimes_{\Gamma,\Gamma}} & \mathbf{NFre}_{\Gamma\times\Gamma}  }
$$

A \emph{multiplicative structure} in a $(\Gamma,\epsilon)$-space
$(B,\epsilon)$ is a morphism $(*,\epsilon):\otimes_{\Gamma}(B,\epsilon)\rightarrow F_{\Sigma}(B,\epsilon)$,
i.e, a family $*_{ij}:B_{i}\otimes B_{j}\rightarrow B_{\epsilon(i,j)}$
of continuous linear maps, with $i,j\in\Gamma$. A \emph{multiplicative
$\Gamma$-space }is a $(\Gamma,\epsilon)$ space in which a multiplicative
structure has been fixed. A \emph{weak morphism} between two multiplicative
$\Gamma$-spaces $(B,\epsilon,*)$ and $(B',\epsilon',*')$ is an
arbitrary morphism in the arrow category of $\mathbf{NFre}_{\Gamma\times\Gamma}$,
i.e, it is given by morphisms $(\zeta,\psi):\otimes_{\Gamma}(B,\epsilon)\rightarrow\otimes_{\Gamma}(B',\epsilon')$
and $(\eta,\nu):F_{\Sigma}(B,\epsilon)\rightarrow F_{\Sigma}(B,\epsilon)$
in the category $\mathbf{NFre}_{\Gamma\times\Gamma}$ such that the
diagrams below commute.
\begin{center}

$
\xymatrix{\ar[d]_{\zeta _{ij}} B_i \otimes B_j \ar[rr]^{*_{ij}} && B_{\epsilon (i,j)} \ar[d]^{\eta_{\epsilon (i,j)}} \\
(B'\otimes B')_{\psi (i,j)} \ar[rr]_-{*'_{\psi (i,j)}} && B'_{\epsilon' (\psi (i,j))}}
$ \quad
$
\xymatrix{\ar[d]_{\psi} \Gamma  \times \Gamma \ar[r]^{\epsilon} & \Gamma \ar[d]^{\nu} \\
\Gamma \times \Gamma \ar[r]_-{\epsilon '} & \Gamma}
$

\par\end{center}

A \emph{strong morphism} (or simply \emph{morphism}) between multiplicative
$\Gamma$-spaces is a weak morphism such that there exists $(\xi,\mu):(B,\epsilon)\rightarrow(B',\epsilon')$
for which $(\zeta,\psi)=\otimes_{\Gamma}(\xi,\mu)$ and $(\eta,\nu)=F_{\Sigma}(\xi,\mu)$.
In explicit terms this means that $\nu=\mu$, $\psi=\mu\times\mu$,
$\zeta=\xi\otimes\xi$ and $\eta=\xi$, so that the diagrams above
become the diagrams below. Notice that the second diagram only means
that $(\xi,\mu)$ is a morphism of $(\Gamma,\epsilon)$-spaces. We
will denote by $\mathbf{NFre}_{\Gamma,*}$ the category of multiplicative
$\Gamma$-spaces and strong morphisms.
\begin{center}

$
\xymatrix{\ar[d]_{\xi _i \otimes \xi _j} B_i \otimes B_j \ar[rr]^{*_{ij}} && B_{\epsilon (i,j)} \ar[d]^{\xi _{\epsilon (i,j)}} \\
B'_{\mu (i)} \otimes B'_{\mu(j)} \ar[rr]_-{*'_{\mu (i) \mu (j)}} && B'_{\epsilon' (\mu (i) \mu (j))}}
$ \quad
$
\xymatrix{\ar[d]_{\mu \times \mu} \Gamma  \times \Gamma \ar[r]^{\epsilon} & \Gamma \ar[d]^{\mu} \\
\Gamma \times \Gamma \ar[r]_-{\epsilon '} & \Gamma}
$

\par\end{center}

An \emph{additive $\Gamma$-space} is a multiplicative $\Gamma$-space
$(B,+,\epsilon)$ such that $\epsilon_{ii}=i$ and such that $+_{ii}:B_{i}\otimes B_{i}\rightarrow B_{i}$
is the addition $+_{i}$ in $B_{i}$. Let $\mathbf{NFre}_{\Gamma,+}\subset\mathbf{NFre}_{\Gamma,*}$
denote the full subcategory of additive $\Gamma$-spaces. The inclusion
has a left-adjoint $F(*,\epsilon)=(+,\delta)$ given by 
\[
\delta(i,j)=\begin{cases}
\epsilon(i,j), & i\neq j\\
i & j=i
\end{cases}\quad\text{and}\quad+_{ij}=\begin{cases}
*_{ij}, & i\neq j\\
+_{i} & j=i
\end{cases}.
\]

\begin{example}
Any $\Gamma$-space $B=(B_{i})_{i\in\Gamma}$ admits many trivial
additive structures, defined as follows. Let $\theta:\Gamma\times\Gamma\rightarrow\Gamma$
be any function and define an additive structure by 
\[
\delta_{\theta}(i,j)=\begin{cases}
\theta(i,j), & i\neq j\\
i & j=i
\end{cases}\quad\text{and}\quad+_{ij}=\begin{cases}
0_{\theta(i,j)}, & i\neq j\\
+_{i} & j=i
\end{cases},
\]
where $0_{\theta(i,j)}$ is the function constant in the zero vector
of $B_{\theta(i,j)}$. All these trivial structures are actually strongly
isomorphic. Indeed, let $B_{\theta}$ denote $B$ with the trivial
additive structure defined by $\theta$. As one can see, we have a
natural bijection 
\[
\operatorname{IMor}_{\Gamma,+}(B_{\theta},B'_{\theta'})\simeq\operatorname{IMor}_{\Gamma}(B,B'),
\]
where the right-hand side is the set of \emph{$\mu$-injective morphisms
}between $B$ and $B'$, i.e, morphisms $(\xi,\mu)$ of $\Gamma$-spaces
such that $\mu:\Gamma\rightarrow\Gamma$ is injective. Furthermore,
this bijection preserve isomorphisms, so that we also have a bijection
\[
\operatorname{IIso}_{\Gamma,+}(B_{\theta},B'_{\theta'})\simeq\operatorname{IISo}_{\Gamma}(B,B').
\]
Since for $B'=B$ the right-hand side contains at least the identity,
it follows that $B_{\theta}\simeq B_{\theta'}$ for every $\theta$
and $\theta'$.
\begin{example}
\label{first_example_additive}Suppose that $\Gamma$ has a partial
order $\leq$. Any increasing or decreasing sequence of nuclear Fr\'echet
spaces, i.e, such that $B_{i}\subset B_{j}$ if either $i\leq j$
or $j\leq i$, has a canonical additive structure, where $+_{ij}$
is given by the sum in $B_{\max(i,j)}$ or $B_{\min(i,j)}$, respectively,
so that $\delta(i,j)=\max(i,j)$ or $\delta(i,j)=\min(i,j)$. In particular,
for any open set $U\subset\mathbb{R}^{n}$ and any order-preserving
integer function $p:\Gamma\rightarrow[0,\infty]$, the sequences $B_{i}=L^{p(i)}(U)$
and $B_{i}'=C^{p(i)}(U)$ are decreasing, where in $L^{p(i)}(U)$
we take the Banach structure given by the $L^{p(i)}$-norm and in
$C^{p(i)}(U)$ we consider Fr\'echet family of seminorms 
\begin{equation}
\Vert f\Vert_{r,l}=\sup_{\vert\mu\vert=r}\sup_{x\in K_{l}}\Vert\partial^{\mu}f(x)\Vert,\label{Frechet_C^k}
\end{equation}
where $0\leq r\leq p(i)$ and $(K_{l})$ is any countable sequence
of compacts such that every other compact $K\subset U$ is contained
in $K_{l}$ for some $l$ \citep{treves}. Thus, both sequences become
an additive $\Gamma$-space in a natural way. They will be respectively
denoted by $L^{p}(U)$ and $C^{p}(U)$. We will be specially interested
in the function $p(i)=k-\alpha(i)$, for $i\leq k$, where $\alpha$
is some other function. In this case, we will write $C^{k-\alpha}(U)$
instead of $C^{p}(U)$. More precisely, we will consider the sequence
given by $B_{i}$=$C^{k-\alpha(i)}(U)$, if $\alpha(i)\leq k$, and
$B_{i}=0$, if $\alpha(i)\geq k$. In this situation, $\delta(i,j)=k-\max(\alpha(i),\alpha(j))$.
\begin{example}
Similarly, if $U\subset\mathbb{R}^{n}$ is a nice open set such that
the Sobolev embeddings are valid \citep{key-1}, then for any fixed
integers $p,q,r>0$ with $p<q$, we have a continuous embedding $W^{r,p}(U)\hookrightarrow W^{\ell(r),q}(U)$,
where $\ell(r)=r+n\frac{(p-q)}{pq}$ and $W^{k,p}(U)$ is the Sobolev
space. Thus, by defining $B_{i}=W^{l(i),q}(U)$, where $l(i)=\ell^{i}(r)$,
i.e, $l(0)=r$, $l(1)=\ell(r)$ and $l(i)=\ell(\ell(...\ell(r)...))$,
we get again a decreasing sequence of Banach spaces and therefore
an additive $\mathbb{Z}_{\geq0}$-space.
\begin{example}
\label{first_example_multiplicative} As in Example \ref{first_example_additive},
suppose $\Gamma$ ordered and consider an integer function $p:\Gamma\rightarrow[0,\infty]$.
In any open set $U\subset\mathbb{R}^{n}$ pointwise multiplication
of real functions give us bilinear maps $\cdot_{ij}:C^{p(i)}(U)\times C^{p(j)}(U)\rightarrow C^{\min(p(i),p(j))}(U),$
which are continuous in the Fr\'echet structure (\ref{Frechet_C^k}),
so that with $\delta(i,j)=\min(p(i),p(j))$, the $\mathbb{Z}_{\geq0}$-space
$C^{p}(U)$ is multiplicative. If $p(i)=k-\alpha(i)$, then $\delta(i,j)=k-\max(\alpha(i),\alpha(j))$.
\begin{example}
\label{Holder_multiplicative} From H\"older's inequality, pointwise
multiplication of real functions also defines continuous bilnear maps
$\cdot_{ij}:L^{i}(U)\times L^{j}(U)\rightarrow L^{i\cdot j/(i+j)}(U)$
for $i,j\geq2$ such that $i\star j=i\cdot j/(i+j)$ is an integer
\citep{key-1,treves}. Let $\mathbb{Z}_{S}$ be the set of all integers
$m\geq2$ such that for every two given $m,m'\in\mathbb{Z}_{S}$ the
sum $m+m'$ divides the product $m\cdot m'$. Suppose $\Gamma$ ordered
and choose a function $p:\Gamma\rightarrow\mathbb{Z}_{S}$. Then $L^{p}(U)$
has a multiplicative structure with $\epsilon(i,j)=p(i)\star p(j)$.
Even if $i\star j$ is not an integer we get a multiplicative structure.
Indeed, let $\epsilon'(i,j)=\operatorname{int}(\epsilon(i,j))$, where
$\operatorname{int}(k)$ denotes the integer part of a real number.
Then $\epsilon'(i,j)\le\epsilon(i,j)$, so that $L^{\epsilon(i,j)}(U)\subset L^{\epsilon'(i,j)}(U)$,
and we can assume $\cdot_{ij}$ as taking values in $L^{\epsilon'(i,j)}(U)$.
\begin{example}
\label{convolution_multiplicative} Given $1\leq i,j\le\infty$, notice
that $i\star j\geq1/2$, which is equivalent to saying the number
$r(i,j)=i\cdot j/(i+j-i\cdot j)$, i.e, the solution of 
\[
\frac{1}{i}+\frac{1}{j}=\frac{1}{r}+1,
\]
also satisfies $1\leq r\leq\infty$. Thus, from Young's inequality,
convolution product defines a continuous bilinear map \citep{key-1,treves}
\[
*_{ij}:L^{i}(\mathbb{R}^{n})\times L^{j}(\mathbb{R}^{n})\rightarrow L^{r(i,j)}(\mathbb{R}^{n}),
\]
so that for any function $p:\Gamma\rightarrow[1,\infty]$ the sequence
$L^{p}(\mathbb{R}^{n})$ has a multiplicative structure with $\epsilon(i,j)=r(p(i),p(j))$.
\begin{example}
A \emph{Banach $\Gamma$-space }is a $\Gamma$-space $B=(B_{i})_{i\in\Gamma}$
such that each $B_{i}$ is a Banach space. Suppose that $\Gamma$
is actually a monoid $(\Gamma,+,0)$. Notice that a multiplicative
structure in $B$ with $\epsilon(i,j)=i+j$ is a family of continuous
bilinear maps $*_{ij}:B_{i}\times B_{j}\rightarrow B_{i+j}$. Since
the category of Banach spaces and continuous linear maps has small
coproducts, the coproduct $\mathcal{B}=\oplus_{i\in\Gamma}B_{i}$
exists as a Banach space and $(\mathcal{B},*)$ is actually a $\Gamma$-graded
normed algebra.
\end{example}
\end{example}
\end{example}
\end{example}
\end{example}
\end{example}
\end{example}
We say that two multiplicative structures $(*,\epsilon)$ and $(+,\delta)$
on the same $\Gamma$-space $B$ are \emph{left-compatible }if the
following diagrams are commutative. The first one makes sense precisely
because the second one is commutative ($\sigma$ is the map $\sigma(x,(y,z))=((x,y),(x,z))$).
A \emph{left-distributive $\Gamma$-space }is a $\Gamma$-space endowed
with two left-compatible multiplicative structures. Morphisms are
just morphisms of $\Gamma$-spaces which preserve both multiplicative
structures. Let $\mathbf{Fre}_{\Gamma,*}^{l}$ be the category of
all of them. Our main interest will be when $(+,\delta)$ is actually
an additive structure, justfying the notation.
\begin{center}
$
\xymatrix{\ar[d]_{id \otimes +_{jk}} B_i \otimes (B_j \otimes B_k) \ar[r]^-{\sigma} & (B_i \otimes B_j)\otimes (B_i \otimes B_k) \ar[r]^-{*_{ij} \otimes *_{ik}} & B_{\epsilon (i,j)} \otimes B_{\epsilon (i,k)} \ar[d]^{+_{\epsilon (i,j)\epsilon(i,k)}} \\
B_i \otimes B_{\delta (j,k)} \ar[rr]_-{*_{i\delta (j,k)}} && B_{\epsilon (i, \delta (j,k))} }
$
\par\end{center}

\begin{center}
$
\xymatrix{\ar[d]_{id \times \delta} \Gamma \times (\Gamma \times \Gamma) \ar[r]^-{\sigma} & (\Gamma \times \Gamma) \times (\Gamma \times \Gamma) \ar[r]^-{\epsilon \times \epsilon} & \Gamma \times \Gamma \ar[d]^{\delta} \\
\Gamma \times \Gamma \ar[rr]_{\epsilon} && \Gamma}
$
\par\end{center}

In a similar way, we say that $(*,\epsilon)$ and $(+,\delta)$ are
\emph{right-compatible} if the diagrams below commute, with $\sigma((x,y),z)=((x,z),(y,z))$.
A \emph{right-distributive }$\Gamma$\emph{-space} is a $\Gamma$-space
together with a right-compatible structure. Morphisms are morphisms
of $\Gamma$-spaces preserving those structures. Let $\mathbf{Fre}_{\Gamma,*}^{r}$
be the corresponding category. Finally, let $\mathbf{Fre}_{\geq,+,*}=\mathbf{Fre}_{\Gamma,+,*}^{l}\cap\mathbf{Fre}_{\Gamma,+,*}^{r}$
be the category of \emph{distributive $\Gamma$-spaces}, i.e, the
category of $\Gamma$-spaces endowed with two multiplicative structures
which are both left-compatible and right-compatible.
\begin{center}
$
\xymatrix{\ar[d]_{+_{ij} \otimes id} (B_i \otimes B_j) \otimes B_k \ar[r]^-{\sigma} & (B_i \otimes B_k)\otimes (B_j \otimes B_k) \ar[r]^-{*_{ik} \otimes *_{jk}} & B_{\epsilon (i,k)} \otimes B_{\epsilon (j,k)} \ar[d]^{+_{\epsilon (i,k)\epsilon(j,k)}} \\
B_{\delta (i,j)} \otimes B_{k} \ar[rr]_-{*_{\delta (i,j)k}} && B_{\epsilon (\delta (i,j),k)} }
$
\par\end{center}

\begin{center}
$
\xymatrix{\ar[d]_{\delta \times id} (\Gamma \times \Gamma) \times \Gamma \ar[r]^-{\sigma} & (\Gamma \times \Gamma) \times (\Gamma \times \Gamma) \ar[r]^-{\epsilon \times \epsilon} & \Gamma \times \Gamma \ar[d]^{\delta} \\
\Gamma \times \Gamma \ar[rr]_{\epsilon} && \Gamma}
$
\par\end{center}
\begin{example}
\label{first_example_distributive} The additive and multiplicative
structures for $C^{p}$ described in Example \ref{first_example_additive}
and Example \ref{first_example_multiplicative} define a distributive
structure. Any $\Gamma$-space $B$, when endowed with a trivial additive
structure and the induced multiplicative structure, becomes a distributive
$\Gamma$-space.
\begin{example}
\label{second_example_distributive} The additive structure in $L^{p}$
given in Example \ref{first_example_additive} is generally \emph{not}
left/right-compatible with the multiplicative structures of Example
\ref{Holder_multiplicative} and Example \ref{convolution_multiplicative}.
Indeed, as one can check, the sum $+$ in $L^{p}$ is compatible with
the pointwise multiplication iff the function $p:\Gamma\rightarrow\mathbb{Z}_{\geq0}$
is such that $p(i)\star p(j)\leq p(i)\star p(k)$ whenever $p(j)\leq p(k)$,
where $i,j,k\in\Gamma$. For fixed $p$ we can clearly restrict to
the subset $\Gamma_{p}$ in which the desired condition is satisfied,
so that $(L^{p(i)})_{i\in\Gamma_{p}}$ is distributive for every $p$.
Simlarly, the $+$ is compatible with the convolution product $*$
in $L^{p(\cdot)}$ iff $r(p(i),p(j))\leq r(p(i),p(k))$ whenever $p(j)\leq p(k)$.
\end{example}
\end{example}
$\quad\;\,$A \emph{$\Gamma$-ambient }is a pair $(\mathbf{X},\gamma)$,
where $\mathbf{X}$ is a category with pullbacks and $\gamma$ is
an embedding $\gamma:\mathbf{NFre}_{\Gamma}\hookrightarrow\mathbf{X}$.
Let $X\in\mathbf{X}$ and consider the corresponding slice category
$\gamma(\mathbf{NFre}_{\Gamma})/X$, i.e, the category of morphisms
$\gamma(B)\rightarrow X$ in $\mathbf{X}$, with $B$ a $\Gamma$-space,
and commutative triangles with vertex $X$. Let $\operatorname{Sub}_{\Gamma}(X,\gamma)$
be the full subcategory whose objects are monomorphisms $\imath:\gamma(B)\hookrightarrow X$.
If $(\gamma(B),\imath)$ belongs to $\operatorname{Sub}_{\Gamma}(X,\gamma)$
we say that it is a \emph{$\Gamma$-subspace} of $(X,\gamma)$. Finally,
let $\operatorname{Span}_{\Gamma}(X,\gamma)$ be the category of spans
$\gamma(B)\hookrightarrow X\hookleftarrow\gamma(B')$ of $\Gamma$-subspaces
of $(X,\gamma)$. If $(\gamma(B),\imath,\gamma(B'),\imath')$ is a
span, we say that $(\mathbf{X},\gamma,X,\imath,\imath'$) is an \emph{intersection
structure between $B$ and $B'$ in $X$} and we define the corresponding
\emph{intersection in $X$} as the object $B\cap_{X,\gamma}B'\in\mathbf{X}$
given by the pullback between $\imath$ and $\imath'$; the object
$X$ itself is called the \emph{base object }for the intersection.
We say that a span in $\operatorname{Span}_{\Gamma}(X,\gamma)$ is
\emph{proper} if the corresponding pullback actually belongs to $\mathbf{NFre}_{\Gamma}$,
i.e, if there exists some $L\in\mathbf{NFre}_{\Gamma}$ such that
$\gamma(L)\simeq B\cap_{X,\gamma}B'$. Notice that, since $\gamma$
is an embedding, when $L$ exists it is unique up to isomorphisms.
Thus, we will write $B\cap_{X}B'$ in order to denote any object in
the isomorphism class. We will also require that the universal maps
$u_{\gamma}:B\cap_{X,\gamma}B'\rightarrow\gamma(B)$ and $u'_{\gamma}:B\cap_{X,\gamma}B'\rightarrow\gamma(B')$
also induce maps $u:B\cap_{X}B'\rightarrow B$ and $u':B\cap_{X}B'\rightarrow u'$,
which clearly exist if the embedding $\gamma$ is full. Let $\operatorname{PSpan}_{\Gamma}(X,\gamma)$
be the full subcategory of proper spans.\begin{equation}{\label{intersection_1} \xymatrix{B \cap_{X,\gamma} B' \ar[r] \ar[d] & \gamma(B') \ar@{^(->}[d]^{\imath '} & B_{\epsilon} \cap_{X,\gamma_{\Sigma}} B'_{\epsilon '} \ar[r] \ar[d] & \gamma_{\Sigma}(B'_{\epsilon '}) \ar@{^(->}[d]^{\jmath '}  \\
\gamma(B) \ar@{^(->}[r]_{\imath } & X & \gamma_{\Sigma}(B_{\epsilon }) \ar@{^(->}[r]_{\jmath } & X}}
\end{equation}

Recall the functor $F_{\Sigma}:\mathbf{NFre}_{\Gamma,\Sigma}\rightarrow\mathbf{NFre}_{\Gamma\times\Gamma}$
assigning to each $(\Gamma,\epsilon)$-space $(B,\epsilon)$ the corresponding
$(\Gamma\times\Gamma)$-space $B\circ\epsilon\equiv B_{\epsilon}$.
Let $(\mathbf{X},\gamma_{\Sigma})$ be a $(\Gamma\times\Gamma)$-ambient.
Let $\operatorname{Sub}_{\Gamma,\Sigma}(X,\gamma_{\Sigma})$ be the
category of monomorphisms $\jmath:F_{\Sigma}(\gamma_{\Sigma}(B,\epsilon))\hookrightarrow X$
for a fixed $X$, i.e, the category of $(\Gamma,\epsilon)$-subspaces
of $X$. The \emph{intersection} \emph{in $X$ }between two of them
is the second pullback above, where for simplicity we write $\gamma_{\Sigma}(B)\equiv F_{\Sigma}(\gamma_{\Sigma}(B,\epsilon))$.
The intersection is \emph{proper} if the resulting object belongs
to $\mathbf{NFre}_{\Gamma,\Sigma}$. Let $\operatorname{PSpan}_{\Gamma,\Sigma}(X,\gamma_{\Sigma})$
be the category of those spans. Now, let $(B,*,\epsilon)$ and $(B',*',\epsilon')$
be two multiplicative $\Gamma$-spaces. An \emph{intersection structure}
between them is a tuple $\mathbb{X}=(\mathbf{X},\gamma_{\Sigma},X,\jmath,\jmath')$
consisting of a $(\Gamma\times\Gamma)$-ambient $(\mathbf{X},\gamma_{\Sigma})$,
and object $X$ and a span $(\jmath,\jmath')$, as in (\ref{intersection_1}).
The \emph{intersection object} between $*$ and $*'$ in $\mathbb{X}$
is the pullback $\operatorname{pb}(*,*';X,\gamma_{\Sigma})$ below.
By universality we get the dotted arrow $*\cap*'$. We say that an
intersection structure is \emph{proper }if not only the span $(\jmath,\jmath')$
is proper, but also the intersection object belongs to $\mathbf{NFre}_{\Gamma,\Sigma}$
(in the same previous sense). In this case, the object in $\mathbf{NFre}_{\Gamma,\Sigma}$
will be denoted by $\operatorname{pb}(*,*';X)$, its components by
$\operatorname{pb}_{ij}(*,*';X)$ and we will define the \emph{intersection
number function }between $*$ and $*'$ in $\mathbb{X}$ as the function
$\#_{*,*';\mathbb{X}}:\Gamma\times\Gamma\rightarrow[0,\infty]$ given
by $\#_{*,*';\mathbb{X}}(i,j)=\dim_{\mathbb{R}}\operatorname{pb}_{ij}(*,*';X)$.
We say that two multiplicative $\Gamma$-spaces $(B,*,\epsilon)$
and $(B',*',\epsilon')$ (or that $*$ and $*'$) have \emph{nontrivial
intersection in $\mathbb{X}$ }if $\#_{*,*';\mathbb{X}}\geq1$. Finally,
we say that $*$ and $*'$ are \emph{nontrivially intersecting }if
they have nontrivial intersection in some intersection structure.
$$
\xymatrix{\operatorname{pb}(*,*';X,\gamma_{\Sigma}) \ar@{-->}[rd]^{*\cap *'} \ar[rr] \ar[dd]  && \gamma _{\Sigma} (B'_{\epsilon '} \otimes B'_{\epsilon '}) \ar[d]^{\gamma _{\Sigma}(*')} \\
 & B_{\epsilon} \cap_{X,\gamma_{\Sigma}} B'_{\epsilon '} \ar[d] \ar[r]  & \gamma _{\Sigma}(B'_{\epsilon})   \ar@{^(->}[d]^{\jmath '} \\
  \gamma _{\Sigma} (B_{\epsilon} \otimes B_{\epsilon}) \ar[r]_{\gamma _{\Sigma} (*)} & \gamma _{\Sigma}(B_{\epsilon})  \ar@{^(->}[r]_-{\jmath}  & X}
$$
\begin{example}
\label{vectorial_intersection_structure}Let $(\mathbf{X},\gamma)$
be any $\Gamma$-ambient such that $\mathbf{X}$ has coproducts. Given
$B,B'$, let $X=\gamma(B)\oplus\gamma(B')$ and notice that we have
monomorphisms $\imath:\gamma(B)\rightarrow X$ and $\imath':\gamma(B')\rightarrow X$,
so that we can consider the intersection $B\cap_{X,\gamma}B'$ having
the coproduct as a base object. However, this is generally trivial.
For instance, if $\mathbf{X}=\mathbf{Set}_{\Gamma}$ is the category
of $\Gamma$-sets, i.e, sequences $(X_{i})_{i\in\Gamma}$ of sets,
then the intersection above is the empty $\Gamma$-set, i.e, $X_{i}=\varnothing$
for each $i\in\Gamma$. Furthermore, if $\mathbf{X}$ is some category
with null objects which is freely generated by $\Gamma$-sets, i.e,
such that there exists a forgetful functor $U:\mathbf{X}\rightarrow\mathbf{Set}_{\Gamma}$
admitting a left-adjoint, then the intersection $B\cap_{X,\gamma}B'$
is a null object. In particular, if $\mathbf{X}=\mathbf{Vec}_{\mathbb{R},\Gamma}$
is the category of $\Gamma$-graded real vector spaces, with $\operatorname{Span}_{\Gamma}$
denoting the left-adjoint, then the intersection object $\operatorname{Span}_{\Gamma}(B\cap_{X,\gamma}B')$
is the trivial $\Gamma$-graded vector space. Let us say that an intersection
structure $\mathbb{X}$ is \emph{vectorial} if is proper and defined
in a $\Gamma$-ambient $(\mathbf{Vec}_{\mathbb{R},\Gamma},\gamma)$
such that $\gamma$ create null-objects (in other words, $B$ is the
trivial $\Gamma$-space iff $\gamma(B)$ is the trivial $\Gamma$-vector
space, i.e, iff $\gamma(B)_{i}\simeq0$ for each $i$). In this case,
it then follows that for the base object $X=\gamma(B)\oplus\gamma(B')$
we have $B\cap_{X,\gamma}B'\simeq0$ when regarded as a $\Gamma$-space.
\begin{example}
\label{standard_intersection_structure} The intersection between
two $\Gamma$-spaces in a vectorial intersection structure is not
necessarily trivial; it actually depends on the base object $X$.
Indeed, suppose that $B$ and $B'$ are such that there exists a $\Gamma$-set
$S$ for which $B\subset S$ and $B'\subset S$. Let $X=\operatorname{Span}_{\Gamma}(B\cup B')$,
where the union is defined componentwise. We have obvious inclusions
and the corresponding intersection object is given by $B\cap_{X,\gamma}B'\simeq\operatorname{Span}_{\Gamma}(B\cap B')$,
where the right-hand side is the intersection as $\Gamma$-vector
spaces, defined componentwise, which is nontrivial if $B\cap B'$
is nonempty as $\Gamma$-sets. For instance, $C^{p}(U)\cap_{X,\gamma}C^{q}(U)$
and $C^{p}(U)\cap_{X,\gamma}L^{q}(U)$ are nontrivial in them for
every $p,q$. We will say that a vectorial intersection structure
with base object $X=\operatorname{Span}_{\Gamma}(B\cup B')$ is a
\emph{standard intersection structure}. Notice that two standard intersection
structures differ from the choice of $\gamma$ and from the maps $\imath$
and $\imath'$ which define the span in $X$.
\end{example}
\end{example}
Vectorial intersection structures have the following good feature:
\begin{prop}
\label{prop_vectorial_structure} Let $(B,\epsilon,*)$ and $(B',\epsilon',*')$
be multiplicative structures and let $\mathbb{X}$ be a vectorial
intersection structure between them. If at least $B$ or $B'$ is
nontrivial, then the intersection space $\operatorname{pb}(*,*';X)$
is nontrivial too, indenpendently of the base object $X$. In other
words, nontrivial multiplicative structures have nontrivial intersection
in any vectorial intersection structure.
\end{prop}
\begin{proof}
Since $\gamma_{\Sigma}$ creates null objects, it is enough to work
in the category of $\Gamma$-vector spaces. On the other hand, since
a $\Gamma$-vector space $V=(V_{i})_{i\in\gamma}$ is nontrivial iff
at least one $V_{i}$ is nontrivial, it is enough to work with vector
spaces. Thus, just notice that if $T:V\otimes W\rightarrow Z$ and
$T':V'\otimes W'\rightarrow Z$ are arbitrary linear maps, then the
pullback between them contains a copy of $V\oplus W\oplus V'\oplus V'$.
\end{proof}
\begin{cor}
\label{corollary_vectorial_structure} Nontrivial multiplicative structures
are always nontrivially intersecting.
\end{cor}
\begin{proof}
Straightforward from the last proposition and from the fact that vectorial
intersection structures exist.
\end{proof}
\begin{rem}
The proposition explains that the intersection space $\operatorname{pb}(*,*';X)$
can be nontrivial even if the intersection $B_{\epsilon}\cap_{X}B'_{\epsilon'}$
is trivial.
\end{rem}
Let $(B,*,+,\epsilon,\delta)$ and $(B',*',+',\epsilon',\delta')$
be two distributive $\Gamma$-spaces. A \emph{distributive intersection
structure} between them is an intersection structure $\mathbb{X}_{*,*'}$
between $*$ and $*'$ together with an intersection structure $X_{+,+'}$
between $+$ and $+'$ whose underlying $\Gamma_{\Sigma}$-ambient
are the same. In other words, it is a tuple $\mathbb{X}_{*,+}=(\mathbf{X},\gamma_{\Sigma},X_{*},X_{+},\jmath_{*},\jmath'_{*},\jmath_{+},\jmath'_{+})$
such that $(X_{*},\jmath_{*},\jmath'_{*})$ and $(X_{+},\jmath_{+},\jmath'_{+})$
are spans as in (\ref{intersection_1}). A \emph{full intersection
structure} between distributive $\Gamma$-spaces is a pair $\mathbb{X}=(\mathbb{X}_{0},\mathbb{X}_{*,+})$,
where $\mathbb{X}_{*,+}=(\mathbb{X}_{*},\mathbb{X}_{+})$ is a distributive
intersection structure and $\mathbb{X}_{0}=(\mathbf{X}_{0},\gamma,X,\imath,\imath')$
is an intersection structure between the $\Gamma$-spaces $B$ and
$B'$, whose underlying ambient category $\mathbf{X}_{0}$ is equal
to the ambient category $\mathbf{X}$ of $\mathbb{X}_{*}$ and $\mathbb{X}_{+}$.
We say that the triple $(\mathbf{X},\gamma,\gamma_{\Sigma})$ is a
\emph{full $\Gamma$-ambient} and that $(X,X_{*},X_{+})$ is the \emph{full
base object}.\emph{ }We also say that $\mathbb{X}$ is \emph{vectorial}
if both $\mathbb{X}_{0}$, $\mathbb{X}_{*}$ and $\mathbb{X}_{+}$
are vectorial. The \emph{full} \emph{intersection space} of the distributive
structures $B$ and $B'$ in the full intersection structure $\mathbb{X}$
is the triple consisting of the intersection spaces $B\cap_{X,\gamma}B'$
in $\mathbb{X}_{0}$, $\operatorname{pb}(*,*';X_{*},\gamma_{\Sigma})$
in $\mathbb{X}_{*}$ and $\operatorname{pb}(+,+';X_{+},\gamma_{\Sigma})$
in $\mathbb{X}_{+}$, denoted simply by $B\cap_{\mathbb{X}}B'$. It
is \emph{nontrivial} (and in this case we say that $B$ and $B'$
have \emph{nontrivial intersection in $\mathbb{X}$}) if each of the
three componentes are nontrivial. When $\mathbb{X}$ is vectorial,
the object representing $B\cap_{\mathbb{X}}B'$ in $\mathbf{NFre}_{\Gamma}\times\mathbf{NFre}_{\Gamma,\Sigma}\times\mathbf{NFre}_{\Gamma,\Sigma}$
will also be denoted by $B\cap_{\mathbb{X}}B'$, i.e, 
\[
B\cap_{\mathbb{X}}B'=(B\cap_{X}B',\operatorname{pb}(*,*';X_{*}),\operatorname{pb}(+,+';X_{+}))
\]

In some situations we will need to work with a special class of ambient
categories $\mathbf{X}$ which becomes endowed with a monoidal structure
that has a nice relation with some closure operator. We finish this
section introducing them. Let $(\mathbf{X},\circledast,1)$ be a monoidal
category, let $\operatorname{Ar}(\mathbf{X})$ be the arrow category
and recall that we have two functors $s,t:\operatorname{Ar}(\mathbf{X})\rightarrow\mathbf{X}$
which assign to each map $f$ its source $s(f)$ and target $t(f)$.
A \emph{weak closure operator} in $\mathbf{X}$ is a functor $\operatorname{cl}:\operatorname{Ar}(\mathbf{X})\rightarrow\operatorname{Ar}(\mathbf{X})$
such that $t(\operatorname{cl}(f))=t(t)$. If $s(f)=B$ and $t(f)=X$,
we write $\operatorname{cl}_{X}(B)$ to denote $s(\operatorname{cl}(f))$.
A \emph{weak closure structure} is a pair $(\operatorname{cl},\mu)$
given by a weak closure operator endowed with a natural transformation
$\mu:id\Rightarrow\operatorname{cl}$, translating the idea of embedding
a space onto its closure. The monoidal product $\circledast$ in $\mathbf{X}$
induces a functor $\overline{\circledast}$ in the arrow category
given $f\circledast f'$ on objects and as below on morphisms. $$
\xymatrix{X \ar[d]_f \ar[r]^{h_x} & X' \ar[d]^{f'}  \ar@{}[d]^{\quad \;\;\overline{\circledast}\;\;} & Z \ar[d]_g \ar[r]^{l_z} & Z' \ar[d]^{g'} \ar@{}[d]^{\quad \;\; \txt{\Large{=}}\quad } & \quad X\circledast Z \ar[d]_{f\circledast g} \ar[r]^{h_x \circledast l_z} & X'\circledast Z' \ar[d]^{f'\circledast g'}  \\
Y \ar[r]_{h_y} & Y'& W \ar[r]_{l_w} & W' & \quad Y\circledast W \ar[r]_{h_y \circledast l_w} & Y'\circledast W'}
$$

A \emph{lax closure operator} in $\mathbf{X}$ is a weak closure operator
which is lax monoidal relative to $\overline{\circledast}$. This
means that for any pair of arrows $(f,f')$ we have a corresponding
arrow morphism $\phi_{f,f'}:\operatorname{cl}(f)\circledast\operatorname{cl}(f')\rightarrow\operatorname{cl}(f\circledast f')$
which is natural in $f,f'$. In particular, if $s(f)=B$, $s(f')=B'$,
$t(f)=X$ and $t(f')=X'$, we have a morphism $\phi:\operatorname{cl}_{X}(B)\circledast\operatorname{cl}_{X'}(B')\rightarrow\operatorname{cl}_{X\circledast X'}(B\circledast B')$.
A \emph{lax closure structure} is a weak closure structure $(\operatorname{cl},\mu)$
whose weak closure operator is actually a lax closure operator $(\operatorname{cl},\phi)$
such that the first diagram below commutes, meaning that $\phi$ and
$\mu$ are compatible. If $f$ and $f'$ have source/target as above,
we have in particular the second diagram.\begin{equation}{\label{closure_1}
\xymatrix{ & \operatorname{cl}(f)\circledast \operatorname{cl}(f') \ar[d]^{\phi_{f,f'}} &&& \operatorname{cl}_{X}(B)\circledast \operatorname{cl}_{X'}(B') \ar[d]^{\phi_{B,B'}}  \\
\ar[ru]^{\mu_f \circledast \mu_{f'}} f\circledast f' \ar[r]_-{\mu_{f\circledast f'}} & \operatorname{cl}(f\circledast f') & B\circledast B' \ar[rru]^{\mu_B \circledast \mu_{B'}} \ar[rr]_-{\mu_{B\circledast B'}}&& \operatorname{cl}_{X\circledast X'}(B\circledast B')}}
\end{equation}

A \emph{monoidal $\Gamma$-ambient }is a $\Gamma$-ambient $(\mathbf{X},\gamma)$
whose ambient category $\mathbf{X}$ is a monoidal category $(\mathbf{X},\circledast,1)$
such that $\gamma:\mathbf{NFre}_{\Gamma}\rightarrow\mathbf{X}$ is
a strong monoidal functor in the sense of \citep{maclane,aguiar},
i.e, lax and oplax monoidal in a compatible way. Thus, for any two
$\Gamma$-spaces $B,B'$ we have an isomorphism $\psi_{B,B'}:\gamma(B)\circledast\gamma(B')\simeq\gamma(B\otimes B')$.
Let $\mathbf{X}$ be a monoidal $\Gamma$-ambient $(\mathbf{X},\gamma,\circledast)$.
A \emph{closure structure} for a $\Gamma$-space $B$ in $\mathbf{X}$
is a lax closure structure $(\operatorname{cl},\mu,\phi)$ in $\mathbf{X}$
such that any morphism $f:\gamma(B)\rightarrow X$ (not necessarily
a subobject) admits an extension $\hat{\mu}_{f}$ relative to $\mu_{f}:\gamma(B)\rightarrow\operatorname{cl}_{X}(\gamma(B))$
as in the diagram below. A \emph{$\Gamma$-space with closure in $\mathbf{X}$}
is $\Gamma$-space $B$ endowed with a closure structure $\mathbf{c}=(\operatorname{cl},\mu,\phi)$
in $\mathbf{X}$. Notice that the diagrams above make perfect sense
in the category $\mathbf{NFre}_{\Gamma,\Sigma}$ of $(\Gamma,\epsilon)$-spaces,
so that we can also define $(\Gamma,\epsilon)$-spaces with closure
in a monoidal $\Gamma$-ambient $(\mathbf{X},\gamma,\circledast)$.
Let $\mathbf{NFre}_{\Gamma,\Sigma,\mathbf{c}}(\mathbf{X})\subset\mathbf{NFre}_{\Gamma,\Sigma}$
be the full subcategory of those $(\Gamma,\epsilon)$-spaces.\begin{equation}{\label{closure_2}
\xymatrix{& X \\
\ar[ru]^f \gamma (B) \ar[r]_{\mu_f} & \operatorname{cl}_{X}(\gamma (B)) \ar@{-->}[u]_{\hat{\mu}_{f}}}}
\end{equation}

\section{$C_{n,\beta}^{k,\alpha}$-Presheaves \label{sec_C_k_a_b}}
\begin{itemize}
\item Let $\Gamma$ be a set such that $\Gamma\cap\mathbb{Z}_{\geq0}\neq\varnothing$
and let $\Gamma_{\geq0}=\Gamma\cap\mathbb{Z}_{\geq0}$. In the following
we will condiser only $\Gamma_{\geq0}$-spaces.
\end{itemize}
$\quad\;\,$We begin by introducting a presheaf version of the previous
concepts. A \emph{presheaf of $\Gamma_{\geq0}$-spaces in $\mathbb{R}^{n}$}
is just a presheaf $B:\operatorname{Open}(\mathbb{R}^{n})^{op}\rightarrow\mathbf{NFre}_{\Gamma_{\geq0}}$
of $\Gamma_{\geq0}$-spaces. Let $\mathbf{B}_{n}$ be the presheaf
category of them. Given a $\Gamma_{\geq0}$-ambient $(\mathbf{X},\gamma)$,
let $X:\operatorname{Open}(\mathbb{R}^{n})^{op}\rightarrow\mathbf{X}$
be a presheaf, i.e, let $X\in\operatorname{Psh}(\mathbb{R}^{n};\mathbf{X})$.
We say that $B\in\mathbf{B}_{n}^{k}$ is a \emph{subobject of $(X,\gamma)$}
if it becomes endowed with a natural trasformation $\imath:\gamma\circ B\hookrightarrow X$
which is objectwise a monomorphism. As in the previous section, let
$\operatorname{Span}_{\Gamma}(X,\gamma)$ be the category of spans
of those subobjects. We have functors 
\[
\operatorname{Base}:\operatorname{Span}_{\Gamma}(X,\gamma)\rightarrow\operatorname{Psh}(\mathbb{R}^{n};\mathbf{X})\quad\text{and}\quad\operatorname{Pb}:\operatorname{Span}_{\Gamma}(X,\gamma)\rightarrow\operatorname{Psh}(\mathbb{R}^{n};\mathbf{X})
\]
which to each span $(\imath,\imath')$ assigns the \emph{base presheaf}
$\operatorname{Base}(\imath,\imath')=X$ and which evaluate the pullback
of $(\imath,\imath')$, i.e, $\operatorname{Pb}(\imath,\imath')(U)=B(U)\cap_{X(U),\gamma}B'(U)$.
In the following we will write $\operatorname{Pb}(\imath,\imath')\equiv B\cap_{X,\gamma}B'$
whenever there is no risck of confusion. We say that a span $(\imath,\imath')$
is \emph{proper }if it is objectwise proper, i.e, if there exists
$L\in\mathbf{B}_{n}$ such that $\gamma\circ L\simeq B\cap_{X,\gamma}B'$.
If exists, then $L$ is unique up to natural isomorphisms and will
be denoted by $B\cap_{X}B'$. Furthermore, we also demand that there
exists $u:B\cap_{X}B'\Rightarrow B$ and $u':B\cap_{X}B'\Rightarrow B'$
such that $\gamma\circ u\simeq u_{\gamma}$ and $\gamma\circ u'\simeq u'_{\gamma}$.
We call $\mathbb{X}=(\mathbf{X},\gamma,X,\imath,\imath')$ and $B\cap_{X,\gamma}B'$
the \emph{intersection structure presheaf} (ISP) and the \emph{intersection
presheaf }between $B$ and $B'$ in $(X,\gamma)$, respectively. We
say that $B$ and $B'$ have \emph{nontrivial intersection in $\mathbb{X}$}
if objectwise they have nontrivial intersection, i.e, if $(B\cap_{X}B')(U)$
have positive real dimension for every $U$.

If $(B,\epsilon)$ and $(B',\epsilon')$ are now presheaves of $(\Gamma_{\geq0},\epsilon)$-spaces
in $\mathbb{R}^{n}$, i.e, if they take values in $\mathbf{NFre}_{\Gamma_{\geq0},\Sigma}$
instead of in $\mathbf{NFre}_{\Gamma_{\geq0}}$, let $\mathbf{B}_{n,\Sigma}$
be the associated presheaf category. We can apply the same strategy
in order to define the category $\operatorname{Span}_{\Gamma,\Sigma}(X,\gamma_{\Sigma})$
of spans of subobjects of $(X,\gamma_{\Sigma})$. If $\jmath:\gamma_{\Sigma}\circ B_{\epsilon}\Rightarrow X$
and $\jmath':\gamma_{\Sigma}\circ B'_{\epsilon'}\Rightarrow X$ are
two of those spans, we define an \emph{intersection structure presheaf
}between $(B,\epsilon)$ and $(B',\epsilon')$ as the tuple \emph{$\mathbb{X}=(\mathbf{X},\gamma_{\Sigma},X,\jmath,\jmath')$}.
Similarly, pullback and projection onto the base presheaf define functors
\[
\operatorname{Base}_{\Sigma}:\operatorname{Span}_{\Gamma,\Sigma}(X,\gamma_{\Sigma})\rightarrow\operatorname{Psh}(\mathbb{R}^{n};\mathbf{X})\quad\text{and}\quad\operatorname{Pb}_{\Sigma}:\operatorname{Span}_{\Gamma,\Sigma}(X,\gamma_{\Sigma})\rightarrow\operatorname{Psh}(\mathbb{R}^{n};\mathbf{X}).
\]
We will write $\operatorname{Pb}_{\Sigma}(\jmath,\jmath')\equiv B_{\epsilon}\cap_{X,\gamma_{\Sigma}}B'_{\epsilon'}$.
If the span is proper, the representing object in $\mathbf{B}_{n,\Sigma}$
is denoted simply by $B_{\epsilon}\cap_{X}B'_{\epsilon'}$, so that
$\gamma_{\Sigma}\circ(B_{\epsilon}\cap_{X}B'_{\epsilon'})\simeq B_{\epsilon}\cap_{X,\gamma_{\Sigma}}B'_{\epsilon'}$.
When $(B,\epsilon,*)$ and $(B',\epsilon',*')$ are presheaves of
multiplicative $\Gamma_{\geq0}$-spaces, meaning that they take values
in $\mathbf{NFre}_{\Gamma_{\geq0},*}$, we denote their presheaf category
$\mathbf{B}_{n,*}$ and define the \emph{intersection space presheaf
}for an intersection structure presheaf $\mathbb{X}_{*,*'}=(\mathbf{X},\gamma_{\Sigma},X,\jmath,\jmath')$
between $*$ and $*'$ as the presheaf which objectwise is the pullback
below in the category $\operatorname{Psh}(\mathbb{R}^{n};\mathbf{X})$.
In other words, it is $\operatorname{Pb}_{\Sigma}((\jmath\circ\gamma_{\Sigma}\circ*,\jmath'\circ\gamma_{\Sigma}\circ*'))$.
If these spans are proper we denote the representing object in $\mathbf{B}_{n,\Sigma}$
by $\operatorname{pb}(*,*';X)$ and we say that $*$ and $*'$ are
\emph{nontrivially intersecting} in $\mathbb{X}$ if that representing
object has objectwise positive dimension.$$
\xymatrix{\operatorname{pb}(*,*';X,\gamma_{\Sigma})  \ar@{=>}[rr] \ar@{=>}[dd]  && \gamma _{\Sigma} \circ (B'_{\epsilon '} \otimes B'_{\epsilon '}) \ar@{=>}[d]^{\gamma _{\Sigma} \circ *'} \\
 & B_{\epsilon} \cap_{X,\gamma_{\Sigma}} B'_{\epsilon '} \ar@{=>}[d] \ar@{=>}[r]  & \gamma _{\Sigma}\circ B'_{\epsilon}   \ar@{=>}[d]^{\jmath '} \\
  \gamma _{\Sigma} \circ (B_{\epsilon} \otimes B_{\epsilon}) \ar@{=>}[r]_{\gamma _{\Sigma} \circ *} & \gamma _{\Sigma}\circ B_{\epsilon}  \ar@{=>}[r]_-{\jmath}  & X}
$$

Let us now consider presheaves of distributive $\Gamma_{\geq0}$-spaces
in $\mathbb{R}^{n}$, i.e, which assigns to each open subset $U\subset\mathbb{R}^{n}$
a corresponding distributive $\Gamma_{\ge0}$-space. Let $\mathbf{B}_{n,*,+}$
be the presheaf category of them. Define a \emph{full intersection
structure presheaf }(full ISP) as a tuple 
\[
\mathbb{X}=(\mathbf{X},\gamma,\gamma_{\Sigma},X,X_{*},X_{+},\imath,\imath',\jmath_{*},\jmath'_{*},\jmath_{+},\jmath'_{+}),
\]
where $(\mathbf{X},\gamma,\gamma_{\Sigma})$ is a full $\Gamma$-ambient,
$X,X_{*},X_{+}\in\operatorname{Psh}(\mathbb{R}^{n};\mathbf{X})$ are
presheaves and $(\imath,\imath')$, $(\jmath_{*},\jmath'_{*})$ and
$(\jmath_{+},\jmath'_{+})$ are spans of subobjects of $X$, $X_{*}$
and $X_{+}$, respectively. Given $B,B'\in\mathbf{B}_{n,*,+}$ we
say that a full ISP $\mathbb{X}$ is \emph{between $B$ and $B'$}
if $B$ is on the domain of $\imath$, $\jmath_{*}$ and $\jmath_{+}$,
while $B'$ is on the domain of $\imath'$, $\jmath'_{*}$ and $\jmath'_{+}$.
Thus, e.g, $\imath:\gamma\circ B\Rightarrow X$, $\jmath_{*}:\gamma_{\Sigma}\circ B_{\epsilon}\Rightarrow X_{*}$
and $\jmath'_{+}:\gamma_{\Sigma}\circ B'_{\epsilon'}\Rightarrow X_{+}$.
We can also write $\mathbb{X}$ as $\mathbb{X}=(\mathbb{X}_{0},\mathbb{X}_{*},\mathbb{X}_{+})$,
where 
\[
\mathbb{X}_{0}=(\mathbf{X},\gamma,X,\imath,\imath'),\quad\mathbb{X}_{*}=(\mathbf{X},\gamma_{\Sigma},X_{*},\jmath_{*},\jmath'_{*})\quad\text{and}\quad\mathbb{X}_{+}=(\mathbf{X},\gamma_{\Sigma},X_{+},\jmath_{+},\jmath'_{+})
\]
are ISP between $B$ and $B'$, between $*$ and $*'$, and between
$+$ and $+'$, respectively. The \emph{full intersection presheaf
}between $B$ and $B'$ in a full ISP $\mathbb{X}$ is the triple
consisting of the intersection presheaves between $B$ and $B'$ in
$\mathbb{X}_{0}$, between $*$ and $*'$ in $\mathbb{X}_{*}$ and
between $+$ and $+'$ in $\mathbb{X}_{+}$. When $\mathbb{X}$ is
proper, the corresponding full intersection presheaf has a representing
object $B\cap_{\mathbb{X}}B'$ in $\mathbf{B}_{n}\times\mathbf{B}_{n,\Sigma}\times\mathbf{B}_{n,\Sigma}$,
given by 
\[
B\cap_{\mathbb{X}}B'=(B\cap_{\mathbb{X}_{0}}B',\operatorname{pb}(*,*';X_{*}),\operatorname{pb}(+,+';X_{+})).
\]
We say that $B$ and $B'$ have \emph{nontrivial intersection }in
a proper full ISP $\mathbb{X}$ if each of the three presheaves in
$B\cap_{\mathbb{X}}B'$ are nontrivial, i.e, have objectwise positive
real dimension.
\begin{example}
By means of varying $U$ and restricting to $\Gamma_{\geq0}$, for
every fixed $n\geq0$, each multiplicative $\Gamma$-space in Examples
\ref{first_example_additive}-\ref{convolution_multiplicative} defines
a presheaf of multiplicative $\Gamma_{\geq0}$-spaces in $\mathbb{R}^{n}$.
Furthermore, by the same process, from Example \ref{first_example_distributive}
and Example \ref{second_example_distributive} we get presheaves of
distributive $\Gamma_{\geq0}$-structures in $\mathbb{R}^{n}$. Vectorial
ISP can be build by following Example \ref{vectorial_intersection_structure}
and Example \ref{standard_intersection_structure}.
\end{example}
Let us introduce $C_{n,\beta}^{k,\alpha}$-presheaves, which will
be the structural presheaves appearing in the definition of $B_{\alpha,\beta}^{k}$-manifold.
Le $K',B'\in\mathbf{B}_{n}$. An \emph{action} of $K'$ in $B'$ is
just a morphism $\kappa':K'\otimes B'\Rightarrow B'$, where the tensor
product is defined objectwise. A \emph{morphism} between actions $\kappa:K\otimes B\Rightarrow B$
and $\kappa':K'\otimes B'\Rightarrow B'$ is just a morphism $(f,g)$
in the arrow category such that $f=\ell\otimes g$. More precisely,
it is a pair $(\ell,g)$, where $\ell:K\Rightarrow K'$ and $g:B\Rightarrow B'$
are morphisms in $\mathbf{B}_{n}$ such that the diagram below commutes.
Let $\mathbf{Act}_{n}$ be the category of actions and morphisms between
them.$$
\xymatrix{K'\otimes B' \ar@{=>}[r]^-{\kappa '} & B' \\
K\otimes B \ar@{=>}[u]^{\ell \otimes g} \ar@{=>}[r]_-{\kappa} & B \ar@{=>}[u]_{g}}
$$

Given $B,B'\in\mathbf{B}_{n}$, an action $\kappa':K'\otimes B'\Rightarrow B'$
and a proper ISP $\mathbb{X}_{0}$ between $B$ and $B'$, we say
that $B$ is \emph{compatible }with $\kappa'$ in $\mathbb{X}_{0}$
if there exists an action $\kappa:A\otimes B\cap_{\mathbb{X}_{0}}B'\Rightarrow B\cap_{\mathbb{X}_{0}}B'$
in the intersection presheaf $B\cap_{\mathbb{X}_{0}}B'$ and a morphism
$\ell:A\Rightarrow K'$ such that $(\ell,u')$ is a morphism of actions,
where $u':B\cap_{\mathbb{X}_{0}}B'\Rightarrow B'$ is the morphism
such that $\gamma\circ u'\simeq u'_{\gamma}$, existing due to the
properness hypothesis. If it is possible to choose $(A,\ell)$ such
that is objectwise a monomorphism, $B$ is called \emph{injectively
compatible }with $\kappa'$ in $\mathbb{X}_{0}$. Given $B,B'\in\mathbf{B}_{n,*,+}$
we say that $B$ is \emph{compatible with $B'$} in a proper full
ISP $\mathbb{X}=(\mathbb{X}_{0},\mathbb{X}_{*},\mathbb{X}_{+})$ if:
\begin{enumerate}
\item the proper full ISP $\mathbb{X}$ is between them;
\item they have nontrivial intersection in $\mathbb{X}$;
\item there exists an action $\kappa':K'\otimes B'\Rightarrow B'$ in $B'$
such that $B$ is injectively compatible with $\kappa'$ in $\mathbb{X}_{0}$.
\end{enumerate}
\begin{example}
\label{exp_action_presheaf} Suppose that $C$ is a presheaf of increasing
$\Gamma_{\geq0}$-Fr\'echet algebras, i.e, such that $C(U)$ is a
$\Gamma_{\geq0}$-space for which each $C(U)_{i}$ is a Fr\'echet
algebra such that there exists a continuous embedding of topological
algebras $C(U)_{i}\subset C(U)_{j}$ if $i\leq j$. Let $p:\Gamma_{\geq0}\rightarrow\Gamma_{\geq0}$
be any function such that $i\leq p(i)$. For $K'=C$ and $B'=C_{p}$
we have an action $\kappa':K'\otimes B'\Rightarrow B'$ such that
$\kappa'_{U,i}:C(U)_{i}\otimes C(U)_{p(i)}\rightarrow C(U)_{p(i)}$
is obtainned by embedding $C(U)_{i}$ in $C(U)_{p(i)}$ and then using
the Fr\'echet algebra multiplication of $C(U)_{p(i)}$. In other
words, under the hypothesis $C$ is actually a presheaf of multiplicative
$\Gamma_{\geq0}$-spaces, with $\epsilon(i,j)=\max(i,j)$, so that
$\kappa'_{U,i}=*_{i,p(i)}$. An analogous discussion holds for decreasing
presheaves.
\end{example}
$\quad\;\,$As a particular case of the last example, we see that
pointwise multiplication induces an action of the presheaf $C^{k}$,
such that $C^{k}(U)_{i}=C^{k}(U)$ for every $U\subset\mathbb{R}^{n}$
and $i\in\Gamma_{\geq0}$, on the presheaf $C^{k-\beta}$, such that
$C^{k-\beta}(U)_{i}=C^{k-\beta(i)}(U)$, where $\beta:\Gamma_{\geq0}\rightarrow[0,k]$
is some fixed integer function. Given $\alpha:\Gamma_{\geq0}\rightarrow\Gamma_{\ge0}$
and $\beta:\Gamma_{\geq0}\rightarrow[0,k]$, we define a \emph{$C_{n,\beta}^{k,\alpha}$-presheaf}
\emph{in a proper full ISP $\mathbb{X}$} as a presheaf $B\in\mathbf{B}_{n,*,+}$
which such that $B_{\alpha}$ is compatible with $C^{k-\beta}$ in
$\mathbb{X}$. Thus, $B$ is a $C_{n,\beta}^{k,\alpha}$-presheaf
in $\mathbb{X}$ if for every $U\subset\mathbb{R}^{n}$:
\begin{enumerate}
\item \emph{the intersection spaces $B_{\alpha}(U)\cap_{X(U)}C^{k-\beta}(U)$,
etc., have positive real dimension}. This means that at least in some
sence (i.e, internal to $X$) the abstract spaces defined by $B$
have a concrete interpretation in terms of differentiable functions
satisfying some further properties/regularity. Furthermore, under
this interpretation, the sum and the multiplication in $B$ are the
pointwise sum and multiplication of differentiable functions added
of properties/regularity;
\item \emph{there exists a subspace $A(U)\subset C^{k}(U)$ and for every
$i\in\Gamma_{\geq0}$ a morphism $\star_{U,i}$ making commutative
the diagram below}.\emph{ }This means that if we regard the abstract
multiplication of $B$ as pointwise multiplication of functions with
additional properties, then that multiplication becomes closed under
a certain set of smooth functions.\begin{equation}{\label{diagram_star}
\xymatrix{C^{k}(U)\otimes C^{k-\beta(i)}(U) \ar[r]^-{\cdot_{U,i}} & C^{k-\beta(i)}(U) \\
\ar@{^(->}[u] A(U) \otimes B_{\alpha (i)}(U) \cap_{X(U)} C^{k-\beta(i)}(U) \ar@{-->}[r]_-{\star_{U,i} } & B_{\alpha(i)}(U) \cap_{X(U)} C^{k-\beta(i)}(U) \ar@{^(->}[u]}}
\end{equation}
\end{enumerate}
$\quad\;\,$If, in addition, the intersection structure $\mathbb{X}_{0}$
between $B_{\alpha}$ and $C^{k-\beta}(U)$ is such that the image
of each $u':B_{\alpha(i)}\cap_{X(U)}C^{k-\beta(i)}(U)\rightarrow C^{k-\beta(i)}(U)$
is closed, we say that $B$ is a \emph{strong $C_{n,\beta}^{k,\alpha}$-presheaf}
\emph{in $\mathbb{X}$}. Finally we say that $B$ is a \emph{nice
$C_{n,\beta}^{k,\alpha}$-presheaf in $\mathbb{X}$} if it is possible
to choose $A$ such that $\dim_{\mathbb{R}}A(U)\cap C_{b}^{\infty}(U)\geq1$,
where $C_{b}^{\infty}(U)$ is the space of bump functions. Notice
that being nice does not depends on the ISP $\mathbb{X}$.
\begin{example}
From Example \ref{exp_action_presheaf} and Proposition \ref{prop_vectorial_structure}
we see that the presheaf of distributive structures $C^{k-}$ is a
$C_{n,\beta}^{k,\alpha}$-presheaf in the full ISP which is objectwise
the standard vectorial intersection structure of Example \ref{standard_intersection_structure},
with $A=C^{\infty}$. Since $C_{b}^{\infty}(U)\subset C^{\infty}(U)$
for every $U$, we conclude that $C^{k-}$ is actually a nice $C_{n,\beta}^{k,\alpha}$-presheaf.
\begin{example}
The same arguments of the previous example can be used to show that
the presheaf $L$ of distributive $\Gamma_{\geq0}$-spaces $L(U)_{i}=L^{i}(U)$,
endowed with the distributive structure given by pointwise addition
and multiplication, is a nice $C_{n,\beta}^{k,\alpha}$-presheaf in
the standard ISP, for $A(U)=C_{b}^{\infty}(U)$ or $A(U)=\mathcal{S}(U)$,
where $\mathcal{S}(U)$ is the Schwarz space \citep{key-1}. An analogous
conclusion is valid if we replace pointwise multiplication with convolution
product.
\end{example}
\end{example}
Let $\mathbf{C}_{n,\beta}^{k,\alpha}$ be the category whose objects
are pairs $(B,\mathbb{X})$, where $\mathbb{X}$ is a proper full
ISP and $B\in\mathbf{B}_{n,*,+}$ is a presheaf of distributive $\Gamma_{\geq0}$-spaces
which is a $C_{n,\beta}^{k,\alpha}$-presheaf in $\mathbb{X}$. Morphisms
$\xi:(B,\mathbb{X})\Rightarrow(B',\mathbb{X}')$ are just morphisms
$\xi:B\Rightarrow B'$ in $\mathbf{B}_{n,*,+}$. We have a projection
$\pi_{n}:\mathbf{C}_{n,\beta}^{k,\alpha}\rightarrow\mathbf{Cat}$
assigning to each pair $(B,\mathbb{X})$ the ambient category $\mathbf{X}$
in the full $\Gamma$-ambient $(\mathbf{X},\gamma,\gamma_{\Sigma})$
of $\mathbb{X}$, where $\mathbf{Cat}$ is the category of all categories.
Notice that such projection really depends only on $n$ (and not on
$\alpha$, $\beta$ and $k$). Let $\mathbf{Cat}_{n}$ be the image
of $\pi_{n}$ and for each $\mathbf{X}\in\mathbf{Cat}_{n}$ let $\mathbf{C}_{n,\beta}^{k,\alpha}(\mathbf{X})$
be the preimage $\pi_{n}^{-1}(\mathbf{X})$. Closing the section,
let us study the dependence of the fiber $\pi_{n}^{-1}(\mathbf{X})$
on the variables $k,n$ and $\beta$. We need some background.

Let $\mathbb{X}$ and $\mathbb{Y}$ be two (not necessarily proper)
ISP in a $\Gamma$-ambient $(\mathbf{X},\gamma)$. A \emph{connection
}between them is a transformation $\xi:X\Rightarrow Y$ between the
underlying base presheaves such that for every two presheaves $B,B'$
of $\Gamma$-subspaces for which both $\mathbb{X}$ and $\mathbb{Y}$
are between $B$ and $B'$, the diagram below commutes. Thus, by universality
we get the dotted arrow.\begin{equation}{\label{connection_int_presheaf}
\xymatrix{ \ar@{==>}[rd]^{\overline{\xi}} \ar@{=>}[dd] B \cap _{X,\gamma} \gamma \circ B' \ar@{=>}[rr] &&  \gamma \circ B' \ar@{=>}[dd]^{\jmath'} \ar@{=>}@/^{1cm}/[ddd]^{\ell '} \\
& \ar@{=>}[ld] \ar@{=>}[ru] B \cap _{Y, \gamma} B' \\
\gamma \circ B \ar@{=>}@/_{0.5cm}/[rrd]_{\ell} \ar@{=>}[rr]_{\jmath} && X \ar@{=>}[d] \\
&& Y}}
\end{equation}

Suppose now that $\mathbf{X}$ is a monoidal $\Gamma$-ambient and
that $B$ and $B'$ are presheaves of $(\Gamma,\mathbf{c})$-spaces
in $\mathbf{X}$, i.e, which objectwise belong to $\mathbf{NFre}_{\Gamma,\mathbf{c}}(\mathbf{X})$
(recall the notation in the end of Section \ref{sec_Gamma_spaces}).
Notice that for any connection $\xi:X\Rightarrow Y$ we have the first
commutative diagram below, where $\mathbf{c}=(\operatorname{cl},\mu)$
is the closure structure. We say that a presheaf of $\Gamma$-spaces
$Z$ is \emph{central }for\emph{ $(B,B',\mathbb{X},\mathbb{Y},\mathbf{c})$
}if it becomes endowed with an embedding $\imath:B\cap_{X,\gamma}B'\Rightarrow Z$
and a morphism $\eta$ such that the second diagram below commutes.
\begin{center}
$
\xymatrix{\ar@{=>}[dd]_{\overline{\xi}} B \cap_{X,\gamma} B' \ar@{=>}[rr]^-{\mu \circ \imath_{X}}  && \operatorname{cl}_{X}(B \cap_{X,\gamma} B') \ar@{=>}[dd]^{\operatorname{cl} \circ \overline{\xi} }\\
\\
B\cap_{Y, \gamma} B'\ar@{=>}[rr]_-{\mu \circ \imath_{Y}} && \operatorname{cl}_{Y}(B \cap_{Y,\gamma} B')}
$\quad $
\xymatrix{\ar@{=>}[rd]^{\imath} \ar@{=>}[dd]_{\overline{\xi}} B \cap_{X,\gamma } B' \ar@{=>}[rr]^-{\mu \circ \imath_{X}}  && \operatorname{cl}_{X}(B\cap_{X,\gamma } B') \ar@{=>}[dd]^{\operatorname{cl} \circ \overline{\xi} }\\
 & Z \ar@{==>}[rd]^{\eta}\\
B\cap_{Y, \gamma} B'\ar@{=>}[rr]_-{\mu \circ \imath_{Y}} && \operatorname{cl}_{Y}(B\cap_{Y,\gamma } B'_{\epsilon '})}
$\quad 
\par\end{center}
\begin{rem}
If $\mathbf{X}$ is complete/cocomplete then $\eta$ always exists,
at least up to universal natural transformations. Indeed, notice that
$\eta$ is actually the extension of $\mu_{Y}\circ\overline{\xi}$
by $\imath$ (equivalently, the extension of $\operatorname{cl}\circ\overline{\xi}\circ\mu_{X}$
by $\imath$), so that due to completeness/cocompleteness we can take
the left/right Kan extension \citep{maclane,handbook}.
\begin{rem}
In an analogous way we can define connections between presheaves of
$(\Gamma,\epsilon)$-spaces and central $(\Gamma,\epsilon)$-presheaves
relative to some closure structure. These can also be regarded as
Kan extensions, so that if $\mathbf{X}$ is complete/cocomplete they
will also exist up to universal natural transformations.
\end{rem}
\end{rem}

\section{Theorem A \label{sec_thm_A}}

$\quad\;\,$We say that a ISP $\mathbb{X}=(\mathbf{X},\gamma,X,\imath,\imath'$)
is \emph{monoidal }if the underlying $\Gamma$-ambient $(\mathbf{X},\gamma)$
is monoidal. A full ISP $\mathbb{X}=(\mathbb{X}_{0},\mathbb{X}_{*},\mathbb{X}_{+})$
is \emph{partially monoidal} if the ISP $\mathbb{X}_{0}$ is monoidal.
Fixed monoidal $\Gamma$-ambient $(\mathbf{X},\gamma,\circledast)$
and given an integer function $p:\Gamma_{\geq0}\rightarrow\mathbb{Z}_{\geq0}$
such that $p(i)\leq k-\beta(i)$ for each $i\in\Gamma_{\geq0}$, define
a $p$\emph{-structure }on a $C_{n,\beta}^{k,\alpha}$-presheaf $B$
in a partially monoidal full proper ISP $\mathbb{X}=(\mathbb{X}_{0},\mathbb{X}_{*},\mathbb{X}_{+})$,
as another non-necessarily proper monoidal ISP $\mathbb{Y}_{0}=(\mathbf{X},\gamma,Y)$,
together with a closure structure $\mathbf{c}=(\operatorname{cl},\mu,\phi)$
and a connection $\xi:Y_{0}\Rightarrow X$ such that $Z=B_{\alpha}\cap_{X,\gamma}C^{p}$
is central for $(B_{\alpha},C^{k-\beta},\mathbb{X}_{0},\mathbb{Y}_{0},\mathbf{c})$
relative to the canonical embedding $\imath:B_{\alpha}\cap_{X,\gamma}C^{k-\beta}\hookrightarrow Z$
induced by universality of pullbacks applied to $C^{k-\beta}\hookrightarrow C^{p}$,
which exists due to the condition $p\leq k-\beta$. Define a \emph{$C_{n,\beta;p}^{k,\alpha}$-presheaf
in $\mathbb{X}$} as a $C_{n,\beta}^{k,\alpha}$-presheaf in $B$
in $\mathbb{X}$ endowed with a $p$-structure and let $\mathbf{C}_{n,\beta;p}^{k,\alpha}$
be the full subcategory of them. Furthermore, let $\mathbf{C}_{n,\beta;p}^{k,\alpha}(\mathbf{X}_{\operatorname{full}})\subset\mathbf{C}_{n,\beta}^{k,\alpha}(\mathbf{X})$
the corresponding full subcategories of pairs $(B,\mathbb{X})$ for
which $\pi_{n}(B,\mathbb{X})=\mathbf{X}$ and $\gamma$ is a full
functor.

We can now prove Theorem A. It will be a consequence of the following
more general theorem:
\begin{thm}
\label{prop_embeddings}For any category with pullbacks $\mathbf{X}$,
there are full embeddings
\end{thm}
\begin{enumerate}
\item[(1)] \emph{$\imath:\mathbf{C}_{n,\beta;p}^{k,\alpha}(\mathbf{X}_{\operatorname{full}})\hookrightarrow\mathbf{C}_{n,p}^{k,\alpha}(\mathbf{X})$,
if $p-\beta\leq k$;}
\item[(2)] \emph{$f_{*}:\mathbf{C}_{n,\beta}^{k,\alpha}(\mathbf{X})\hookrightarrow\mathbf{C}_{r,\beta}^{k,\alpha}(\mathbf{X})$,
for any continuous injective map $f:\mathbb{R}^{n}\rightarrow\mathbb{R}^{r}$.}
\end{enumerate}
\begin{proof}
We begin by proving (1). Notice that the fiber $\pi_{n}^{-1}(\mathbf{X})$
is a nonempty category only if $\mathbf{X}$ for every $k,\alpha,\beta$
is part of a partially monoidal proper full ISP, with $\gamma$ full,
in which at least one $C_{n,\beta}^{k,\alpha}$-presheaf is defined.
Since for empty fibers the result is obvious, we assume the above
condition. Thus, given $(B,\mathbb{X})\in\mathbf{C}_{n,\beta;p}^{k,\alpha}(\mathbf{X}_{\operatorname{full}})$
(which exists by the assumption on $\mathbf{X}$), we will show that
it actually belongs to $\mathbf{C}_{n,p}^{k,\alpha}(\mathbf{X})$.
Since we are working with full subcategories this will be enough for
(1). We assert that $B_{\alpha}$ and $C^{p}$ have nontrivial intersection
in $\mathbb{X}$. Since the functor $\gamma$ creates null-objects,
it is enough to prove that $B_{\alpha}\cap_{X,\gamma}C^{p}$, $(B_{\alpha})_{\epsilon}\cap_{X_{*},\gamma_{\Sigma}}(C^{p})_{\epsilon'}$
and $(B_{\alpha})_{\delta}\cap_{X_{*},\gamma_{\Sigma}}(C^{p})_{\delta'}$.
But, since $\delta'=\epsilon'=\min$ and since $p\leq\beta-k$, from
universality and stability of pullbacks under monomorphisms we get
monomorphisms $B_{\alpha}\cap_{X,\gamma}C^{k-\beta}\hookrightarrow B_{\alpha}\cap_{X,\gamma}C^{p}$,
etc. Now, being $B$ a $C_{n,\beta}^{k,\alpha}$-presheaf in $\mathbb{X},$the
left-hand sides $B_{\alpha}\cap_{X,\gamma}C^{k-\beta}$, etc., are
nontrivial, which implies $B_{\alpha}$ and $C^{p}$ have nontrivial
intersection in $\mathbb{X}$. By the same arguments we get the commutative
diagram below, where $A$ is the presheaf arising from the $C_{n,\beta}^{k,\alpha}$-structure
of $B$. Our task is to extend $\star_{\beta}$ by replacing $k-\beta$
with $p$. If $p(i)-\beta(i)=k$ there is nothing to do for such $i$.
Thus, suppose $p(i)-\beta(i)<k$ for all $i$. More precisely, our
task is to get the dotted arrow in the second diagram below, where
the long vertical arrows arise from the universality of pullbacks,
as above. We use simple arrows intead of double arrows in order to
simplify the notation.

$
\xymatrix{ C^{\infty}\otimes C^p \ar[r]^-{\cdot} & C^p  \\
\ar@{^(->}[u] C^{\infty}\otimes C^{k-\beta} \ar[r]^-{\cdot} & C^{k- \beta} \ar@{^(->}[u] \\
\ar@{^(->}[u] A \otimes (B_\alpha \cap_X C^{k- \beta}) \ar[r]_-{\star_{\beta}} & B_\alpha \cap_{X}  C^{k- \beta} \ar@{^(->}[u]}
$$
\xymatrix{ \ar@{^(->}[d] A\otimes (B_{\alpha} \cap_X C^p) \ar@{-->}[r]^{\star _p} & \ar@{^(->}[d] B_{\alpha} \cap_X C^p  \\ 
 C^{\infty}\otimes C^p \ar[r]^-{\cdot} & C^p  \\
\ar@{^(->}[u] C^{\infty}\otimes C^{k-\beta} \ar[r]^-{\cdot} & C^{k- \beta} \ar@{^(->}[u] \\
 \ar@/^{1.2cm}/[uuu] \ar@{^(->}[u] A \otimes (B_\alpha \cap_X C^{k- \beta}) \ar[r]_-{\star_{\beta}} & B_\alpha \cap_{X}  C^{k- \beta} \ar@{^(->}[u] \ar@/_{1.2cm}/[uuu]}
$

\noindent Since $(B,\mathbb{X})\in\mathbf{C}_{n,\beta;p}^{k,\alpha}(\mathbf{X}_{\operatorname{full}})$,
there exists a ISP $\mathbb{Y}_{0}$, a connection $\xi:Y\Rightarrow X$
and a closure structure $\mathbf{c}=(\operatorname{cl},\mu,\phi)$
such that $B_{\alpha}\cap_{X,\gamma}C^{p}$ is central for $(B_{\alpha},C^{k-\beta},\mathbb{X}_{0},\mathbb{Y}_{0},\mathbf{c})$.
Since any morphism extends to the closure and recalling that $\gamma$
is strong monoidal, we have the diagram below. It commutes due to
the commutativity of (\ref{closure_1}) and (\ref{closure_2}) and
due to the naturality of $\psi$. We are also using that $\gamma\circ B\cap_{X}B'\simeq B\cap_{X,\gamma}B'$.
Furthermore, $\mu_{A}$ is $\mu_{\gamma(\imath_{A})}$, where $\imath_{A}:A\hookrightarrow C^{\infty}$
is the embedding. Again we use simple arrows instead of double arrows.$$
\xymatrix{\ar[d]_{\phi} \operatorname{cl}_{\gamma (C^{\infty})}(\gamma (A))\circledast \operatorname{cl}_{X}(B_\alpha \cap _{X,\gamma} C^{k- \beta}) & \ar[l]_{id \circledast (\operatorname{cl}\circ \overline{\xi})} \operatorname{cl}_{\gamma(C^{\infty})}(\gamma (A))\circledast \operatorname{cl}_{Y}(B_{\alpha} \cap _{Y,\gamma} C^{k-\beta} ) \\
 \ar[r]^-{\hat{\mu}_{\gamma (\star _{\beta})\circ \psi}} \operatorname{cl}_{\gamma (C^{\infty}) \circledast X} (\gamma (A) \circledast (B_{\alpha} \cap _{X,\gamma} C^{k- \beta})) & B_\alpha \cap_{X,\gamma}  C^{k- 
\beta} \\
\ar[d]<0.1cm>^{\psi ^{-1}} \ar[ru]^{\gamma (\star _\beta )} \gamma (A\otimes (B_{\alpha} \cap _{X} C^{k-\beta})) \ar[r]_-{\gamma (id \otimes \imath) } & \gamma (A\otimes (B_{\alpha} \cap _{X} C^p)) \ar[d]<0.1cm>^{\psi ^{-1}}   \\
\ar@/^{2cm}/[uu]^{\mu_{A\circledast X}}  \ar[u]<0.1cm>^{\psi} \gamma (A) \circledast (B_{\alpha} \cap _{X,\gamma} C^{k-\beta}) \ar[r]_-{\gamma (id) \circledast \gamma (\imath) } & \ar[u]<0.1cm>^{\psi} \gamma (A)\circledast (B_{\alpha} \cap _{X,\gamma} C^p)  \ar@/_{2cm}/[uuu]_{\mu_{A} \circledast \eta}  } 
$$Let us consider the composition morphism 
\[
f_{\beta}^{p}=\hat{\mu}_{\gamma(\star_{\beta})\circ\psi}\circ\phi\circ(id\circledast(\operatorname{cl}\circ\overline{\xi}))\circ(\mu_{A}\circledast\eta)\circ\psi^{-1}:\gamma(A\otimes B_{\alpha}\cap_{X}C^{p})\rightarrow\gamma(B_{\alpha}\cap_{X}C^{k-\beta}).
\]
On the other hand, we have $\gamma(\imath_{b,p}):\gamma(B_{\alpha}\cap_{X}C^{k-\beta})\hookrightarrow\gamma(B_{\alpha}\cap_{X}C^{p})$
arising from the embedding $C^{k-\beta}\hookrightarrow C^{p}$. Composing
them we get a morphism
\begin{equation}
\gamma(\imath_{b,p})\circ f_{\beta}^{p}:\gamma(A\otimes B_{\alpha}\cap_{X}C^{p})\rightarrow\gamma(B_{\alpha}\cap_{X}C^{p}).\label{gamma_is_full}
\end{equation}
Since $\gamma$ is fully-faithful, there exists exactly one $\star_{p}:A\otimes B_{\alpha}\cap_{X}C^{p}\rightarrow B_{\alpha}\cap_{X}C^{p}$
such that $\gamma(\star_{p})=\gamma(\imath_{b,p})\circ f_{\beta}^{p}$,
which is our desired map. That it really extends $\star_{\beta}$
follows from the commutativity of all diagrams involved in the definition
of $\star_{p}$. For (2), recall that any continuous map $f:\mathbb{R}^{n}\rightarrow\mathbb{R}^{r}$
induces a pushforward functor $f^{-1}:\mathbf{Psh}(\mathbb{R}^{n})\rightarrow\mathbf{Psh}(\mathbb{R}^{r})$
between the corresponding presheaf categories of presheaves of sets,
which becomes an embedding when $f$ is injective \citep{sheaves}.
Since $(f^{-1}F)(U)=F(f^{-1}(U))$, it is straighforward to verify
that if $(B,\mathbb{X},A)$ belongs to $\mathbf{C}_{n,\beta}^{k,\alpha}(\mathbf{X})$,
then $(f^{-1}B,f^{-1}A,f^{-1}\mathbb{X})\in\mathbf{C}_{r,\beta}^{k,\alpha}(\mathbf{X})$,
where $f^{-1}\mathbb{X}$ is defined componentwise. Thus, we have
an injective on objects functor $f^{-1}:\mathbf{C}_{n,\beta}^{k,\alpha}(\mathbf{X})\rightarrow\mathbf{C}_{r,\beta}^{k,\alpha}(\mathbf{X})$.
Because we are working with full subcategories, it follows that $f^{-1}$
is full and therefore a full embedding.
\end{proof}
\begin{cor}
\label{corollary_embedding}Let $\mathbf{C}_{n,\beta;l}^{k,\alpha}(\mathbf{X}_{\operatorname{full}})$
and $\mathbf{C}_{n,\beta;\beta'}^{k,\alpha}(\mathbf{X}_{\operatorname{full}})$
be the category $\mathbf{C}_{n,\beta;p}^{k,\alpha}(\mathbf{X}_{\operatorname{full}})$
for $p_{l}(i)=l-\beta(i)$ and $p_{\beta'}(i)=k-\beta'(i)$, respectively.
Then there are full embeddings
\end{cor}
\begin{itemize}
\item $\jmath:\mathbf{C}_{n,\beta;\beta'}^{k,\alpha}(\mathbf{X}_{\operatorname{full}})\hookrightarrow\mathbf{C}_{n,\beta}^{k,\alpha}(\mathbf{X}_{\operatorname{full}})$,
if $\beta'\leq\beta$;
\item $\imath:\mathbf{C}_{n,\beta;l}^{k,\alpha}(\mathbf{X}_{\operatorname{full}})\hookrightarrow\mathbf{C}_{n,\beta}^{l,\alpha}(\mathbf{X}_{\operatorname{full}})$,
if $l\leq k$.
\end{itemize}
\begin{proof}
Just notice that if $\beta'\leq\beta$ implies $p_{\beta'}-\beta\leq k$
and that $l\le k$ implies $p_{l}-\beta\leq k$, and then uses (1)
of Theorem \ref{prop_embeddings}.
\end{proof}
\noindent \textbf{Theorem A.} \emph{There are full embeddings}
\begin{enumerate}
\item[(1)] \emph{$\imath:\mathbf{C}_{n,\beta;l}^{k,\alpha}\hookrightarrow\mathbf{C}_{n,\beta}^{l,\alpha}$,
if $l\leq k$;}
\item[(2)] \emph{$f_{*}:\mathbf{C}_{n,\beta}^{k,\alpha}\hookrightarrow\mathbf{C}_{r,\beta}^{k,\alpha}$,
for any continuous injective map $f:\mathbb{R}^{n}\rightarrow\mathbb{R}^{r}$;}
\item[(3)] \emph{$\jmath:\mathbf{C}_{n,\beta;\beta'}^{k,\alpha}\hookrightarrow\mathbf{C}_{r,\beta'}^{k,\alpha}$,
if $\beta'\leq\beta$.}$\underset{\underset{\;}{\;}}{\;}$
\end{enumerate}
\begin{proof}
Straighforward from Corollary \ref{corollary_embedding} and condition
(2) of Theorem \ref{prop_embeddings}.
\end{proof}
\begin{rem}
The requirement of $\gamma$ being full is a bit strong. Indeed, our
main examples of $\Gamma$-ambients are the vectorial ones. But requiring
a full embedding $\gamma:\mathbf{NFre}_{\Gamma}\rightarrow\mathbf{Vec}_{\mathbb{R},\Gamma}$
is a really strong condition. When looking at the proof of Theorem
\ref{prop_embeddings} we see that the only time when the full hypothesis
on $\gamma$ was needed is to ensure that the morphism (\ref{gamma_is_full})
in $\operatorname{Psh}(\mathbb{R}^{n};\mathbf{X})$ is induced by
a morphism in $\mathbf{B}_{n}$. Thus, the hypothesis of being full
can be clearly weakened. Actually, the hypothesis on $\gamma$ being
an embedding can also be weakened. In the end, the only hypothesis
needed in order to develop the previous results is that $\gamma$
creates null-objects and some class of monomorphisms. Since from now
on we will not focus on the theory of $C_{n,\beta}^{k,\alpha}$-presheaves
itself, we will not modify our hypothesis. On the other hand, in futures
works concerning the study of the categories $\mathbf{C}_{n,\beta}^{k,\alpha}$
this refinement on the hypothesis will be very welcome.
\end{rem}

\section{$B_{\alpha,\beta}^{k}$-Manifolds \label{sec_B_k,a,b_manifolds}}

$\quad\;\,$Let $\mathbb{X}$ be a proper full intersection structure
and let $B\in\mathbf{C}_{n,\beta}^{k,\alpha}$ be a $C_{n,\beta}^{k,\alpha}$-presheaf
in $\mathbb{X}$. A $C^{k}$-function
$f:U\rightarrow\mathbb{R}^{m}$ is called a \emph{$(B,k,\alpha,\beta)$-function
in $(\mathbb{X},m)$ }if $\partial^{\mu}f_{j}\in B_{\alpha(i)}(U)\cap_{X(U)}C^{k-\beta(i)}(U)$
for all $\vert\mu\vert=i$ and $i\in\Gamma_{\geq0}$. Due to the compatibility
between the operations of $B_{\alpha}$ and $C^{k-\beta}$ at the
intersection, it follows that the collection $B_{\alpha,\beta}^{k}(U;\mathbb{X},m)$
of all $(B,k,\alpha,\beta)$-functions in $(\mathbb{X},m)$ is a real
vector space. This will become more clear in the next proposition.
First, notice that by varying $U\subset\mathbb{R}^{n}$ we get a presheaf
(at least of sets) $B_{\alpha,\beta}^{k}(-;\mathbb{X},m)$. Recall
that a \emph{strong} $C_{n,\beta}^{k,\alpha}$-presheaf is one in
which $B_{\alpha}\cap_{X}C^{k-\beta}\hookrightarrow C^{k-\beta}$
is objectwise closed.
\begin{prop}
\label{prop_compatible_is_Frechet}For every $B\in\mathbf{C}_{n,\beta}^{k,\alpha}$,
every $\mathbb{X}$ and every $m\geq0$, the presheaf of $(B,k,\alpha,\beta)$-functions
in $(\mathbb{X},m)$ is a presheaf of real vector spaces. If $B$
is strong, then it is actually a presheaf of nuclear Fr\'echet spaces.
\end{prop}
\begin{proof}
Consider the following spaces:
\[
\mathbb{C}_{\alpha,\beta}^{k}(U,m)=\prod_{\kappa(m)}C^{k-\beta(i)}(U)\quad\text{and}\quad\mathbb{BC}_{\alpha,\beta}^{k}(U;\mathbb{X},m)=\prod_{\kappa(m)}B_{\alpha(i)}(U)\cap_{X(U)}C^{k-\beta(i)}(U),
\]
where $\prod_{\kappa(m)}=\prod_{j=1}^{m}\prod_{i}\prod_{\vert\alpha\vert=i}\prod_{m^{i}}$.
Since they are countable products of nuclear Fr\'echet spaces, they
have a natural nuclear Fr\'echet structure. Consider the map $j^{k}:C^{k}(U;\mathbb{R}^{m})\rightarrow\mathbb{C}^{k,\alpha}(U,m)$,
given by $j^{k}f=(\partial^{\mu}f_{j})_{\mu,j}$, with $\vert\mu\vert=i$,
and notice that $B_{\alpha,\beta}^{k}(U;\mathbb{X},m)$ is the preimage
of $j^{k}$ by $\mathbb{BC}_{\alpha,\beta}^{k}(U;\mathbb{X},m)$.
The map $j^{k}$ is linear, so that any preimage has a linear structure,
implying that the space of $(B,k,\alpha,\beta)$-functions is linear.
But $\jmath^{k}$ is also continuous in those topologies, so that
if $\mathbb{BC}_{\alpha,\beta}^{k}(U;\mathbb{X},m)$ is a closed subset
in $\mathbb{C}_{\alpha,\beta}^{k}(U,m)$, then $B_{\alpha,\beta}^{k}(U;\mathbb{X},m)$
is a closed subset of a nuclear Fr\'echet spaces and therefore it
is also nuclear Fr\'echet. This is ensured precisely by the strong
hypothesis on $B$.
\end{proof}
\begin{example}
Even if $B$ is not strong, the space of $(B,k,\alpha,\beta)$-functions
may have a good structure. Indeed, let $p\in[1,\infty]$ be fixed,
let $\alpha(i)=p$, $\beta(i)=i$ and let $L$ be the presheaf $L(U)_{i}=L^{i}(U)$,
regarded as a $C_{n,\beta}^{k,\alpha}$-presheaf in the standard ISP.
Thus, $L_{\alpha}(U)\cap_{X(U)}C^{k-\beta}(U)$ is the $\Gamma_{\geq0}$-space
of components $L^{p}(U)\cap C^{k-i}(U)$. A $(B,k,\alpha,\beta)$-function
is then a $C^{k}$-map such that $\partial^{\mu}f_{j}\in L^{p}(U)\cap C^{k-i}(U)$
for every $\vert\mu\vert=i$. Therefore, the space of all of them
is the strong (in the sense of classical derivatives) Sobolev space
$W^{k,p}(U;\mathbb{R}^{m})$, which is a Banach space \citep{key-1}.
But $L^{p}(U)\cap C^{k-i}(U)$ is clearly not closed in $C^{k-i}(U)$,
since sequences of $L^{p}$-integrable $C^{k}$-functions do not necessarily
converge to $L^{p}$-integrable maps \citep{key-1}.
\end{example}
\begin{rem}
In order to simplify the notation, if $m=n$ the space of $(B,k,\alpha,\beta)$-functions
will be denoted by $B_{\alpha,\beta}^{k}(U;\mathbb{X})$ instead of
$B_{\alpha,\beta}^{k}(U;\mathbb{X},n)$.
\end{rem}
Let $M$ be a Hausdorff paracompact topological space. A $n$\emph{-dimensional
$C^{k}$-structure} in $M$ is a $C^{k}$-atlas in the classical sense,
i.e, a family $\mathcal{A}$ of coordinate systems $\varphi_{i}:U_{i}\rightarrow\mathbb{R}^{n}$
whose domains cover $M$ and whose transition functions $\varphi_{j}\circ\varphi_{i}^{-1}:\varphi(U_{ij})\rightarrow\mathbb{R}^{n}$
are $C^{k}$, where $U_{ij}=U_{i}\cap U_{j}$. A \emph{$n$-dimensional
$C^{k}$-manifold} is a pair $(M,\mathcal{A})$, where $\mathcal{A}$
is a maximal $C^{k}$-structure. Given a $C_{n,\beta}^{k,\alpha}$-presheaf
$B$ in a proper full ISP $\mathbb{X}$, define a $(B_{\alpha,\beta}^{k},\mathbb{X})$-\emph{structure
}on a $C^{k}$-manifold $(M,\mathcal{A})$ as a subatlas $\mathcal{B}_{\alpha,\beta}^{k}(\mathbb{X})\subset\mathcal{A}$
such that $\varphi_{j}\circ\varphi_{i}^{-1}\in B_{\alpha,\beta}^{k,n}(\varphi_{i}(U_{ij});\mathbb{X}(\varphi_{i}(U_{ij})))$.
A \emph{$(B_{\alpha,\beta}^{k},\mathbb{X})$-manifold }is one in which
a $(B_{\alpha,\beta}^{k},\mathbb{X})$-structure has been fixed. A
\emph{$(B_{\alpha,\beta}^{k},\mathbb{X})$-morphism} between two $C^{k}$-manifolds
is a $C^{k}$-function $f:M\rightarrow M'$ such that $\phi f\varphi^{-1}\in B_{\alpha,\beta}^{k,n}(\varphi(U);\mathbb{X}(\varphi(U)))$,
for every $\varphi\in\mathcal{B}_{\alpha,\beta}^{k}(\mathbb{X})$
and $\phi\in\mathcal{B}_{\alpha,\beta}'^{k}(\mathbb{X})$. The following
can be easily verified:
\begin{enumerate}
\item A $C^{k}$-manifold $(M,\mathcal{A})$ admits a $(B_{\alpha,\beta}^{k},\mathbb{X})$-structure
iff there exists a subatlas $\mathcal{B}_{\alpha,\beta}^{k}(\mathbb{X})$
for which the identity $id:M\rightarrow M$ is a $(B,k,\alpha,\beta)$-morphism
in $\mathbb{X}$.
\item If $M_{j}$ is a $(B_{\alpha_{j},\beta_{j}}^{k},\mathbb{X})$-manifold,
with $j=1,2$, then there exists a $(B_{\alpha,\beta}^{k},\mathbb{X})$-morphism
between them only if $\beta(i)\geq\max_{j}\{b_{j}(i)\}$ for every
$i$. In particular, if $\beta_{j}(i)=i$, then such a morphism exists
only if $i\leq\beta(i)$.
\end{enumerate}
\begin{example}
In the standard vectorial ISP $\mathbb{X}$ the Example \ref{example_1_int}
and Example \ref{example_2_int} contains the basic examples of $(B_{\alpha,\beta}^{k},\mathbb{X})$-manifolds.
\end{example}
$\quad\;\,$We would like to consider the category $\mathbf{Diff}_{\alpha,\beta}^{B,k}(\mathbb{X})$
of $(B_{\alpha,\beta}^{k},\mathbb{X})$-manifolds with $(B_{\alpha,\beta}^{k},\mathbb{X})$-morphisms
between them. The next lemma reveals, however, that the composition
of those morphisms is not well-defined.
\begin{lem}
Let $B$ be a $C_{n,\beta}^{k,\alpha}$-presheaf in a proper full
ISP $\mathbb{X}$, let $\Pi(i)$ be the set of partitions of $[i]=\{1,..,i\}$.
For each $\mu[i]\in\Pi(i)$, write $\mu[i;r]$ to denote its blocks,
i.e, $\mu[i]=\sqcup_{r}\mu[i;r]$. Let $\leq$ be a function assigning
to each $i\in\Gamma_{\geq0}$ orderings $\leq_{i}$ in $\Pi(i)$ and
$\leq_{\mu[i]}$ in the set of blocks of $\mu[i]$, as follows:
\begin{eqnarray*}
\mu[i]_{\min} & \leq_{i} & \mu[i]_{\min+1}\leq_{i}\cdots\leq_{i}\mu[i]_{\max-1}\leq_{i}\mu[i]_{\max}\\
\mu[i;\min] & \leq_{\mu[i]} & \mu[i;\min+1]\leq_{\mu[i]}\cdots\leq_{\mu[i]}\mu[i;\max-1]\leq_{\mu[i]}\mu[i;\max].
\end{eqnarray*}
Then, for any given open sets $U,V,W\subset\mathbb{R}^{n}$, composition
induces a map\footnote{In the subspace topology this is actually a continuous map, but we
will not need that here.}
\[
\overline{\circ}:B_{\alpha,\beta}^{k,n}(U,V;\mathbb{X})\times B_{\alpha,\beta}^{k,n}(V,W;\mathbb{X})\rightarrow B_{\alpha_{\leq},\beta_{\leq}}^{k,n}(U,W;\mathbb{X}),
\]
where $B_{\alpha,\beta}^{k,n}(U,V;\mathbb{X})$ is the subspace of
$(B,k,\alpha,\beta)$-functions $f:U\rightarrow\mathbb{R}^{n}$ such
that $f(U)\subset V$, and 
\begin{eqnarray*}
\alpha_{\leq}(i) & = & \delta(\overline{\epsilon}_{\mu[i]_{\max}},\delta(\overline{\epsilon}_{\mu[i]_{\max-1}},\delta(\overline{\epsilon}_{\mu[i]_{\max-2}},\cdots,\delta(\overline{\epsilon}_{\mu[i]_{\min+1}},\overline{\epsilon}_{\mu[i]_{\min}})\cdots)))\\
\beta_{\leq}(i) & = & \max_{\mu[i]}\max_{\mu[i;r]}\{\beta(\vert\mu[i]\vert),\beta(\vert\mu[i;r]\vert)\}.
\end{eqnarray*}
Here, if $\mu[i]$ is some partition of $[i]$, then 
\[
\overline{\epsilon}_{\mu[i]}=\epsilon(\alpha(\vert\mu[i]\vert),\epsilon(\mu_{\alpha}[i])),
\]
where for any block $\mu[i;r]$ of $\mu[i]$ we denote $\vert\mu_{\alpha}[i;r]\vert\equiv\alpha(\vert\mu[i;r]\vert)$
and 
\[
\epsilon(\mu_{\alpha}[i])\equiv\epsilon(\vert\mu_{\alpha}[i;\max]\vert,\epsilon(\vert\mu_{\alpha}[i;\max-1]\vert,\vert\mu_{\alpha}[i;\max-2]\vert,\cdots\epsilon(\vert\mu_{\alpha}[i;\min+1]\vert,\vert\mu_{\alpha}[i;\min]\vert)\cdots))).
\]
\end{lem}
\begin{proof}
The proof follows from Fa\`a di Bruno's formula giving a chain rule
for higher order derivatives \citep{key-1} and from the compatibility
between the multiplicative/additive structures of $B_{\alpha}$ and
$C^{k-\beta}$. First of all, notice that if $C^{k}(U;V)$ is the
set of all $C^{k}$-functions $f:U\rightarrow\mathbb{R}^{n}$ such
that $f(U)\subset V$, then for any $g\in C^{k}(V;W)$ the composition
$g\circ f$ is well-defined. Thus, our task is to show that there
exists the dotted arrow making commutative the diagram below.$$
\xymatrix{C^k(U;V)\times C^k(V;W) \ar[r]^-{\circ} & C^k(U;W) \\
\ar@{^(->}[u] B_{\alpha,\beta}^{k}(U,V;\mathbb{X}) \times B_{\alpha,\beta}^{k}(V,W;\mathbb{X}) \ar@{-->}[r]_-{\overline{\circ}} & B_{\alpha,\beta}^{k}(U,W;\mathbb{X}) \ar@{^(->}[u]}
$$Now, let $f\in C^{k}(U;V)$ and $g\in C^{k}(V;W)$, so that from Fa\`a
di Bruno's formula, for any $0\leq i\leq k$ and any multi-index $\mu$
such that $\vert\mu\vert=i$, we have 
\[
\partial^{\mu}(g\circ f)_{j}=\sum_{j}\sum_{\mu[i]}\partial^{\mu[i]}g_{j}\prod_{\mu[i;r]\in\mu[i]}\partial^{\mu[i;r]}f_{j}\equiv\sum_{j}\sum_{\mu[i]}\partial^{\mu[i]}g_{j}f_{j}^{\mu[i]}.
\]
Consequently, if $f$ and $g$ are actually $(B,k,\alpha,\beta)$-functions
in $\mathbb{X}$, then 
\[
\partial^{\mu[i]}g_{j}\in B_{\alpha(\vert\mu[i]\vert)}(V)\cap_{X(V)}C^{k-\beta(\vert\mu[i]\vert)}(V)\quad\text{and}\quad\partial^{\mu[i;r]}f_{j}\in B_{\alpha(\vert\mu[i;r]\vert)}(U)\cap_{X(U)}C^{k-\beta(\vert\mu[i;r]\vert)}(U).
\]
Under the choice of ordering functions $\leq$, from the compatibility
of multiplicative structures we see that
\begin{eqnarray*}
f_{j}^{\mu[i]} & \in & B_{\epsilon(\mu_{\alpha}[i])}(U)\cap_{X(U)}C^{k-\underset{\mu[i;r]}{\max}\{\beta(\vert\mu[i;r]\vert)\}}(U)\\
\partial^{\mu[i]}g_{j}f_{j}^{\mu[i]} & \in & B_{\overline{\epsilon}_{\mu[i]}}(U)\cap_{X(U)}C^{k-\underset{\mu[i;r]}{\max}\{\beta(\vert\mu[i;r]\vert),\beta(\vert\mu[i]\vert)\}}(U).
\end{eqnarray*}
Finally, compatibility of additive structures shows that $\partial^{\mu}(g\circ f)_{j}\in B_{\alpha_{\leq}(i)}(U)\cap_{X(U)}C^{k-\beta_{\leq}(i)}(U)$,
so that by varying $i$ we conclude that $g\circ f$ is a $(B,k,\alpha_{\leq},\beta_{\leq})$-function
in $\mathbb{X}$.
\end{proof}
\begin{rem}
Unlike $B_{\alpha,\beta}^{k}(-;\mathbb{X})$, the rule $U\mapsto B_{\alpha,\beta}^{k}(U,U;\mathbb{X})$
is generally not a presehaf, since the restriction of a map $f:U\rightarrow U$
to an open set $V\subset U$ needs not take values in $V$.
\end{rem}
The problem with the composition can be avoided by imposing conditions
on $B$. Indeed, given a ordering function $\leq$ as above, let us
say that a $C_{\alpha,\beta}^{k,n}$-presheaf $B$ \emph{preserves
$\leq$ }(or that it is \emph{ordered}) \emph{in $\mathbb{X}$} if
there exists an embedding of presheaves $B_{\alpha_{\leq}}\cap_{X}C^{k-\beta_{\leq}}\hookrightarrow B_{\alpha}\cap_{X}C^{k-\beta}$.
\begin{example}
We say that $B$ is \emph{increasing }(resp. \emph{decreasing}) if
for any $U$ we have embeddings $B(U)_{i}\hookrightarrow B(U)_{j}$
(resp. $B(U)_{j}\hookrightarrow B(U)_{i}$) whenever $i\leq j$, where
the order is the canonical order in $\Gamma_{\geq0}\subset\mathbb{Z}_{\geq0}$.
Suppose that $\beta_{\leq}(i)\leq\beta(i)$, that $B$ is increasing
(resp. decreasing) and that $\alpha_{\leq}(i)\leq\alpha(i)$ (resp.
$\alpha_{\leq}(i)\geq\alpha(i)$). Thus, for any $U$ and any $i$
we have embeddings $B(U)_{\alpha_{\leq}(i)}\hookrightarrow B(U)_{\alpha(i)}$
and $B(U)_{\beta_{\leq}(i)}\hookrightarrow B(U)_{\beta(i)}$, so that
by the universality of pullbacks and stability of monomorphisms we
see that any intersection presheaf makes $B$ ordered.
\end{example}
\begin{cor}
With the same notations and hypotheses of the previous lemma, if $B$
is ordered relative to some intersection presheaf $\mathbb{X}$, then
the composition induces a map
\[
\overline{\circ}:B_{\alpha,\beta}^{k}(U,V;\mathbb{X})\times B_{\alpha,\beta}^{k}(V,W;\mathbb{X})\rightarrow B_{\alpha,\beta}^{k}(U,W;\mathbb{X}).
\]
\end{cor}
\begin{proof}
Straightforward.
\end{proof}
On the other hand, we also have problems with the identities: the
identity map $id:U\rightarrow U$ is not necessarily a $(B,k,\alpha,\beta)$-function
in an arbitrary intersection structure $\mathbb{X}$ for an arbitrary
$B$. We say that $B$ is \emph{unital in $\mathbb{X}$} if $id_{U}\in B_{\alpha,\beta}^{k}(U;\mathbb{X)}$
for every open set $U\subset\mathbb{R}^{n}$. Thus, with this discussion
we have proved:
\begin{prop}
\label{category_B_k_a_b} If $B$ is as $C_{n,\beta}^{k,\alpha}$-presheaf
which is ordered and unital in some intersection presheaf $\mathbb{X}$,
then the category $\mathbf{Diff}_{\alpha,\beta}^{B,k}(\mathbb{X})$
is well-defined.
\end{prop}

\section{Theorem B\label{sec_existence}}

$\quad\;\,$In this section we will finally prove an existence theorem
of $(B_{\alpha,\beta}^{k},\mathbb{X})$-structures on $C^{k}$-manifolds
under certain conditions on $B$, meaning absorption and retraction
conditions, culminating in Theorem B. Given a $C_{n,\beta}^{k,\alpha}$-presheaf
$B$ in $\mathbb{X}$ and open sets $U,V\subset\mathbb{R}^{n}$, let
$\operatorname{Diff}_{\alpha,\beta}^{k}(U,V;\mathbb{X})$ be the set
of \emph{$(B,k,\alpha,\beta)$-diffeomorphisms} from $U$ to $V$
in $\mathbb{X}$, i.e, the largest subset for which there exists the
dotted arrow below.$$
\xymatrix{B_{\alpha,\beta}^{k}(U,V;\mathbb{X}) \ar@{^(->}[r] & C^k(U;V) \\
\ar@{^(->}[u] \operatorname{Diff}_{\alpha,\beta}^{k}(U,V;\mathbb{X}) \ar@{-->}[r] & \operatorname{Diff}^{k}(U;V) \ar@{^(->}[u]}
$$

We say that $B$ is \emph{left-absorbing }(resp. \emph{right-absorbing})
\emph{in $\mathbb{X}$} if for every $U,V,W$ there also exists the
dotted arrow in the lower (resp. upper) square below, i.e, $g\circ f$
remains a $(B,k,\alpha,\beta)$-diffeomorphism whenever $f$ (resp.
$g$) is a $(B,k,\alpha,\beta)$-diffeomorphism and $g$ (resp. $f$)
is a $C^{k}$-diffeomorphism. If $B$ is both left-absorbing and right-absorbing,
we say simply that it is \emph{absorbing in $\mathbb{X}$}.$$
\xymatrix{\ar@{^(->}[d]_{id \times \imath} \operatorname{Diff}^{k}(U;V) \times \operatorname{Diff}_{\alpha,\beta}^{k}(V,W;\mathbb{X}) \ar@{-->}[r]^-{\circ _r} & \operatorname{Diff}_{\alpha,\beta}^{k}(U,W;\mathbb{X}) \ar@{^(->}[d] \\
\operatorname{Diff}^{k}(U;V) \times \operatorname{Diff}^{k}(V;W) \ar[r]^-{\circ} & \operatorname{Diff}^{k}(U;W) \\
\ar@{^(->}[u]^{\imath \times id} \operatorname{Diff}_{\alpha,\beta}^{k}(U,V;\mathbb{X}) \times  \operatorname{Diff}^{k}(V;W) \ar@{-->}[r]_-{\circ _l} & \operatorname{Diff}_{\alpha,\beta}^{k}(U,W;\mathbb{X}) \ar@{^(->}[u]}
$$

A more abstract description of these absorbing properties is as follows.
Let $\mathbf{C}$ be an arbitrary category and let $\operatorname{Iso}(\mathbf{C})$
the set of isomorphisms in $\mathbf{C}$, i.e,
\[
\operatorname{Iso}(\mathbf{C})=\coprod_{X,Y\in\operatorname{Ob}(\mathbf{C})}\operatorname{Iso}_{\mathbf{C}}(X;Y).
\]
Let $\operatorname{Com}(\mathbf{C})\subset\operatorname{Iso}(\mathbf{C})\times\operatorname{Iso}(\mathbf{C})$
be the pullback between the source and target maps $s,t:\operatorname{Iso}(\mathbf{C})\rightarrow\operatorname{Ob}(\mathbf{C})$.
Composition gives a function $\circ:\operatorname{Com}(\mathbf{C})\rightarrow\operatorname{Iso}(\mathbf{C})$.
If $\mathbf{C}$ has a distinguished object $*$, we can extend $\circ$
to the whole $\operatorname{Iso}(\mathbf{C})\times\operatorname{Iso}(\mathbf{C})$
by defining $\circ_{*}$, such that $g\circ_{*}f=g\circ f$ when $(f,g)\in\operatorname{Com}(\mathbf{C})$
and $g\circ_{*}f=id_{*}$, otherwise. Thus, $(\operatorname{Iso}(\mathbf{C}),\circ_{*})$
is a magma. Define a \emph{left $*$-ideal }(resp. \emph{right $*$-ideal})
in $\mathbf{C}$ as a map $I$ assigning to each pair of objects $X,Y\in\mathbf{C}$
a subset $I(X;Y)\subset\operatorname{Iso}_{\mathbf{C}}(X;Y)$ such
that the corresponding subset 
\[
I(\mathbf{C})=\coprod_{X,Y\in\operatorname{Ob}(\mathbf{C})}I(X;Y)
\]
of $\operatorname{Iso}(\mathbf{C})$ is actually a left ideal (resp.
right ideal) for the magma structure induced by $\circ_{*}$. A \emph{bilateral
$*$-ideal }(or \emph{$*$-ideal}) in $\mathbf{C}$ is a map $I$
which is both left and right $*$-ideal. When $*$ is an initial object,
we say simply that $I$ is a left-ideal, right-ideal or ideal in $\mathbf{C}$.
\begin{prop}
A $C_{n,\beta}^{k,\alpha}$-presheaf $B$ is left-absorbing (resp.
right-absorbing or absorbing) in a proper full ISP $\mathbb{X}$ iff
the induced rule 
\[
U,V\mapsto\operatorname{Diff}_{\alpha,\beta}^{k,n}(U,V;\mathbb{X})
\]
is a left ideal (resp. right ideal or ideal) in the full subcategory
of $\mathbf{Diff}^{k}$, consisting of open sets of $\mathbb{R}^{n}$
and $C^{k}$-maps between them.
\end{prop}
\begin{proof}
Immediate from the definitions above.
\end{proof}
Notice that in the context of vector spaces, since these are free
abelian objects, the short exact sequence below always split, so that
from the splitting lemma we conclude the existence of a retraction
$r_{U}$, such that $r_{U}\circ\imath=id$, for every $U$ \citep{algebra_homologica}.$$
\xymatrix{0 \ar[r] & B_{\alpha,\beta}^{k}(U;\mathbb{X}) \ar@{^(->}[r]^-{\imath} & \ar@/_{0.5cm}/[l]_{r_U} C^k(U;\mathbb{R}^n) \ar[r]^-{\pi} & C^k(U.\mathbb{R}^n)/B_{\alpha,\beta}^{k}(U;\mathbb{X}) \ar[r] & 0 }
$$

By restriction, for each $V$ we have an induced retraction $r_{U,V}$,
as in the first diagram below. On the other hand, the dotted arrow
does not necessarily exists. In other words, $r_{U,V}$ need not preserve
diffeomorphisms. We say that $B$ has \emph{retractible $(B,k,\alpha,\beta)$-diffeomorphisms
in $\mathbb{X}$} if for every $U,V$ there exist $\overline{r}_{U,V}$
in the second diagram, not necessarily making the first diagram commutative.
A \emph{retraction presheaf in $\mathbb{X}$ }for $B$ is a rule $\overline{r}$
, assigning to each $U,V$ a retraction $\overline{r}_{U,V}$.\begin{equation}{\label{retraction_presheaf}
\xymatrix{B_{\alpha,\beta}^{k}(U,V;\mathbb{X}) \ar@{^(->}[r]^{\imath} & \ar@/_{0.5cm}/[l]_{r_{U,V}} C^k(U;V) &  \operatorname{Diff}_{\alpha,\beta}^{k}(U,V;\mathbb{X}) \ar@{^(->}[r]_-{\imath _{U,V}} & \operatorname{Diff}^{k}(U;V) \ar@{-->}@/_{0.5cm}/[l]_{\overline{r}_{U,V}} \\
\ar@{^(->}[u] \operatorname{Diff}_{\alpha,\beta}^{k}(U,V;\mathbb{X}) \ar@{^(->}[r] & \operatorname{Diff}^{k}(U;V) \ar@{^(->}[u] \ar@{-->}@/_{0.5cm}/[l]_{\overline{r}_{U,V}}} }
\end{equation}

Given a $C^{k}$-manifold $(M,\mathcal{A})$ and a $C_{n,\beta}^{k,\alpha}$-presheaf
$B$ in $\mathbb{X}$, let $C^{k}(\mathcal{A})$ and $B_{\alpha,\beta}^{k,n}(\mathcal{A};\mathbb{X})$
denote the collection of not necessarily maximal $C^{k}$-structures
$\mathcal{A}'\subset\mathcal{A}$ and $(B_{\alpha,\beta}^{k},\mathbb{X})$-structures
$\mathcal{B}_{\alpha,\beta}^{k}(\mathbb{X})\subset\mathcal{A}$, respectively.
Observe that there is an inclusion $\imath_{\mathcal{A}}:B_{\alpha,\beta}^{k}(\mathcal{A};\mathbb{X})\hookrightarrow C^{k}(\mathcal{A})$,
which take a $(B_{\alpha,\beta}^{k},\mathbb{X})$-structure and regard
it as a $C^{k}$-structure. We can now finally prove that for certain
classes of $B$ the set $B_{\alpha,\beta}^{k,n}(\mathcal{A};\mathbb{X})$
is non-empty.
\begin{thm}
\label{existence_thm_1} Let $B$ be a $C_{n,\beta}^{k,\alpha}$-presheaf
which is ordered, left-absorbing or right-absorbing, and which has
retractible $(B,k,\alpha,\beta)$-diffeomorphisms, all of this in
the same proper full ISP $\mathbb{X}$. In this case, for any $C^{k}$-manifold
$(M,\mathcal{A})$, the choice of a retraction presheaf $\overline{r}$
induces a function $\kappa_{\overline{r}}:C^{k}(\mathcal{A})\rightarrow B_{\alpha,\beta}^{k,n}(\mathcal{A};\mathbb{X})$
which is actually a retraction for $\imath_{\mathcal{A}}$. In particular,
under this hypothesis every $C^{k}$-manifold has a $(B_{\alpha,\beta}^{k},\mathbb{X})$-structure.
\end{thm}
\begin{proof}
Let $\mathcal{A}'\subset\mathcal{A}$ be some not necessarily maximal
$C^{k}$-structure and let $\varphi_{i}:U_{i}\rightarrow\mathbb{R}^{n}$
be its charts. The transition functions are given by $\varphi_{ji}=\varphi_{j}\circ\varphi_{i}^{--1}:\varphi_{i}(U_{ij})\rightarrow\varphi_{j}(U_{ij})$.
Notice that the restricted chart $\varphi_{j}\vert_{U_{ij}}$ can
be recovered by $\varphi_{j}\vert_{U_{ij}}=\varphi_{ji}\circ\varphi_{i}\vert_{U_{ij}}$
for each $i$ such that $U_{ij}\neq\varnothing$. This motivate us
to define new functions $\overline{\varphi}_{j;i}:U_{ij}\rightarrow\mathbb{R}^{n}$
by $\overline{\varphi}_{j;i}=\overline{r}(\varphi_{ji})\circ\varphi_{i}\vert_{U_{ij}}$,
where $\overline{r}$ is a fixed restriction presheaf. They are homeomorphisms
onto their images because they are composites of them. We assert that
when varying $i$ and $j$, the maps $\overline{\varphi}_{j;i}$ generate
a $(B_{\alpha,\beta}^{k},\mathbb{X})$-structure, which we denote
by $\overline{r}(\mathcal{A}')$. Indeed, for each $i,j,k,l$, then
transition functions $\overline{\varphi}_{k;l}\circ\overline{\varphi}_{j;i}^{-1}:\overline{\varphi}_{j;i}(U_{ijkl})\rightarrow\overline{\varphi}_{k;l}(U_{ijkl})$
are given by 
\begin{eqnarray*}
\overline{\varphi}_{k;l}\circ\overline{\varphi}_{j;i}^{-1} & = & [\overline{r}(\varphi_{kl})\circ\varphi_{l}\vert_{U_{ijkl}}]\circ[\varphi_{i}^{-1}\vert_{U_{ijkl}}\circ\overline{r}(\varphi_{ji})^{-1}]\\
 & = & \overline{r}(\varphi_{kl})\circ(\varphi_{li})\vert_{U_{ijkl}}\circ\overline{r}(\varphi_{ji})^{-1},
\end{eqnarray*}
which are $(B,k,\alpha,\beta)$-functions in $\mathbb{X}$, due to
the absorbing properties of $B$ in $\mathbb{X}$. More precisely,
if $B$ is left-absorbing, then $\overline{r}(\varphi_{kl})\circ(\varphi_{li})\vert_{U_{ijkl}}$
is a $(B,k,\alpha,\beta)$-function in $\mathbb{X}$. But $\overline{r}(\varphi_{ji})^{-1}$
is also a $(B,k,\alpha,\beta)$-function in $\mathbb{X}$ and by the
ordering hypothesis on $B$ the composite remains a $(B,k,\alpha,\beta)$-function
in $\mathbb{X}$. In this case, define $\kappa_{\overline{r}}(\mathcal{A}')=\overline{r}(\mathcal{A}')$.
If $B$ is right-absorbing, similar argument holds. That $\kappa_{\overline{r}}$
is a retraction for $\imath_{\mathcal{A}}$, i.e, that $\overline{r}(\imath_{\mathcal{A}}(\mathcal{B}_{\alpha,\beta}^{k}(\mathbb{X})))=\mathcal{B}_{\alpha,\beta}^{k}(\mathbb{X})$
follows from the fact that $\overline{r}$ is a retraction presheaf.
\end{proof}
Observe that if $B$ is left-absorbing or right-absorbing, then it
is automatically unital, so that from Proposition \ref{category_B_k_a_b}
under the hypotheses of the last theorem the category $\mathbf{Diff}_{\alpha,\beta}^{B,k}(\mathbb{X})$
is well-defined. We have an obvious forgetful functor $F_{\alpha,\beta}^{B,k}:\mathbf{Diff}_{\alpha,\beta}^{B,k}(\mathbb{X})\rightarrow\mathbf{Diff}^{k}$
which takes a $(B_{\alpha,\beta}^{k}\mathbb{X})$-manifold $(M,\mathcal{A},\mathcal{B}_{\alpha,\beta}^{k}(\mathbb{X}))$
and forgets $\mathcal{B}_{\alpha,\beta}^{k}(\mathbb{X})$ (this is
essentally an extension of the inclusion $\imath_{\mathcal{A}}$).
Our task is to show that this functor has adjoints. We begin by proving
existence of adjoints on the core.

We recall that if $\mathbf{C}$ is a category, then its \emph{core}
is the subcategory $\operatorname{C}(\mathbf{C})\subset\mathbf{C}$
obtained by forgetting all morphisms which are not isomorphisms. Every
functor $F:\mathbf{C}\rightarrow\mathbf{D}$ factors through the core,
so that we have an induced functor $\operatorname{C}(F):\operatorname{C}(\mathbf{C})\rightarrow\operatorname{C}(\mathbf{D})$.
Actually, the core construction provides a functor $\operatorname{C}:\mathbf{Cat}\rightarrow\mathbf{Cat}$,
where $\mathbf{Cat}$ denotes the category of all categories \citep{maclane,handbook}.
\begin{thm}
\label{existence_thm_2} With the same notations and hypotheses of Theorem
\ref{existence_thm_1}, for every restriction presheaf $\overline{r}$
the rule $\kappa_{\overline{r}}$ induces a functor $K_{\overline{r}}:\operatorname{C}(\mathbf{Diff}^{k})\rightarrow\operatorname{C}(\mathbf{Diff}_{\alpha,\beta}^{B,k}(\mathbb{X}))$.
If $B$ is left-absorbing (resp. right-absorbing), then $K_{\overline{r}}$
is a left (resp. right) adjoint for the core of the forgetful functor.
In particular, if $B$ is absorbing, then $\operatorname{C}(F_{\alpha,\beta}^{B,k})$
is ambidextrous adjoint.
\end{thm}
\begin{proof}
Define $K_{\overline{r}}$ by $K_{\overline{r}}(M,\mathcal{A})=(M,\mathcal{A},\overline{r}(\mathcal{A}))$
on objects and by $K_{\overline{r}}(f)=f$ on morphisms. On objects
it is clearly well-defined. On morphisms it is too, because for any
$\overline{\varphi}_{j;i},\overline{\varphi}_{k;l}\in\overline{r}(\mathcal{A})$
we have 
\begin{eqnarray}
\overline{\varphi}_{k;l}\circ f\circ\overline{\varphi}_{j;i}^{-1} & = & [\overline{r}(\varphi_{kl})\circ\varphi_{l}\vert_{U_{ijkl}}]\circ f\circ[\varphi_{i}^{-1}\vert_{U_{ijkl}}\circ\overline{r}(\varphi_{ji})^{-1}]\label{local_charts_function}\\
 & = & \overline{r}(\varphi_{kl})\circ(\varphi_{l}f\varphi_{i})\vert_{U_{ijkl}}\circ\overline{r}(\varphi_{ji})^{-1}.\nonumber 
\end{eqnarray}
Since we are in the core, $(\varphi_{l}f\varphi_{i})$ is a $C^{k}$-diffeomorphism,
so that we can use the same arguments of that used in Theorem \ref{existence_thm_1}
to conclude that (\ref{local_charts_function}) is a $(B,k,\alpha,\beta)$-function
in $\mathbb{X}$. Preservation of compositions and identities is clear,
so that $K_{\overline{r}}$ really defines a functor. Suppose that
$B$ is left-absorbing. Given a $C^{k}$-manifold $(M,\mathcal{A})$
and a $(\text{\ensuremath{B'}}_{\alpha,\beta}^{k},\mathbb{X})$-manifold
$(M',\mathcal{A}',\mathcal{B'}_{\alpha,\beta}^{k}(\mathbb{X}))$ we
assert that there is a bijection\begin{equation}{\label{left_adjunction}
\xymatrix{\operatorname{Diff}_{\alpha,\beta}^{B,k}(K_{\overline{r}}(M,\mathcal{A});(M',\mathcal{A}',\mathcal{B'}_{\alpha,\beta}^{k}(\mathbb{X})) \ar[r]<-0.1cm>_-{\imath _{M,M'}} & \operatorname{Diff}^{k}((M,\mathcal{A});(M',\mathcal{A}')) \ar[l]<-0.1cm>_-{\xi_{M,M'}}  } }
\end{equation}which is natural in both manifolds. Define $\imath_{M,M'}(f)=f$ and
notice that this is well-defined, since locally it is given by the
inclusions $\imath_{U,V}$ in (\ref{retraction_presheaf}). Define
$\xi_{M,M'}(f)=f$. In order to show that this is also well-defined,
let $f:(M,\mathcal{A})\rightarrow(M',\mathcal{A}')$ a $C^{k}$-diffeomorphism
and let $\overline{\varphi}_{j;i}\in\overline{r}(\mathcal{A})$ and
$\phi\in\mathcal{B'}_{\alpha,\beta}^{k}(\mathbb{X})$ charts. Thus,
\begin{equation}
\phi\circ f\circ\overline{\varphi}_{j;i}^{-1}=(\phi\circ f\circ\varphi_{i}^{-1})\circ\overline{r}(\varphi_{ji})^{-1}.\label{local_expression_thm_2}
\end{equation}
Due to the inclusion $\imath_{U,V}$, the chart $\phi$ is a $C^{k}$-chart,
so that $\phi\circ f\circ\varphi_{i}^{-1}$ is a $C^{k}$-diffeomorphism
(since $f$ is a $C^{k}$-diffeomorphism). The left-absorption property
then implies that (\ref{local_expression_thm_2}) is a $(B,k,\alpha,\beta)$-function
in $\mathbb{X}$, meaning that $\xi_{M,M'}$ is well-defined. That
(\ref{left_adjunction}) holds is clear; naturality follows from the
fact that the maps $\imath_{M,M'}$ and $\xi_{M,M'}$ do not depends
on the manifolds. The case in which $B$ is right-absorbing is completely
analogous.
\end{proof}
\begin{cor}
With the same notations and hypotheses of Theorem \ref{existence_thm_1},
the function $\kappa_{\overline{r}}$ independs of $\overline{r}$.
More precisely, if $\overline{r}$ and $\overline{r}'$ are two retraction
presheaves, then there exists a natural isomorphism $K_{\overline{r}}\simeq K_{\overline{r}'}$,
so that for every $C^{k}$-manifold $(M,\mathcal{A})$ we have a corresponding
$(B,k,\alpha,\beta)$-diffeomorphism $(M,\overline{r}(\mathcal{A}))\simeq(M,\overline{r}'(\mathcal{A}))$.
\end{cor}
\begin{proof}
Straighforward from the uniqueness of the left and right adjoints
\citep{maclane,handbook}.
\end{proof}
We would like to extend Theorem \ref{existence_thm_2} to the whole
category $\mathbf{Diff}_{\alpha,\beta}^{B,k}(\mathbb{X})$. In order
to do this, notice that when proving Theorem \ref{existence_thm_2}
the hypothesis that we are working on the core was used only to conclude
that the local expressions $\phi f\varphi$ are $C^{k}$-diffeomorphisms,
leading us to use the diffeomorphism-absorption properties. But, if
instead absorbing only diffeomorphisms we can absorb every $C^{k}$-map,
we will then be able to absorb $\phi f\varphi$ for every $f$, meaning
that the same proof will still work in $\mathbf{Diff}_{\alpha,\beta}^{B,k}(\mathbb{X})$.

We say that a $C_{n,\beta}^{k,\alpha}$-presheaf $B$ is \emph{fullly
left-absorbing }(resp. \emph{fully right-absorbing}) \emph{in $\mathbb{X}$}
if for every $U,V,W$ there exists the dotted arrow in the lower (resp.
upper) square below. If $B$ is both fully left-aborving and fully
right-absorbing, we say simply that it is \emph{fully absorbing in
$\mathbb{X}$}. There is also an abstract characterization in terms
of left/right/bilateral ideals, but now considered in the magma $\operatorname{Mor}(\mathbf{C})$
of all morphisms instead of on the magma $\operatorname{Iso}(\mathbf{C})$
of isomorphisms.$$
\xymatrix{\ar@{^(->}[d]_{id \times \imath} \operatorname{C}^{k}(U;V) \times \operatorname{Diff}_{\alpha,\beta}^{k,n}(V,W;\mathbb{X}) \ar@{-->}[r]^-{\circ _r} & \operatorname{Diff}_{\alpha,\beta}^{k,n}(U,W;\mathbb{X}) \ar@{^(->}[d] \\
\operatorname{C}^{k}(U;V) \times \operatorname{C}^{k}(V;W) \ar[r]^-{\circ} & \operatorname{C}^{k}(U;W) \\
\ar@{^(->}[u]^{\imath \times id} \operatorname{Diff}_{\alpha,\beta}^{k,n}(U,V;\mathbb{X}) \times  \operatorname{C}^{k}(V;W) \ar@{-->}[r]_-{\circ _l} & \operatorname{Diff}_{\alpha,\beta}^{k,n}(U,W;\mathbb{X}) \ar@{^(->}[u]}
$$\textbf{Theorem B.} \emph{If $B$ is ordered, fully left-absorbing
(resp. fully right-absorbing) and has retractible $(B,k,\alpha,\beta)$-diffeomorphisms,
all of this in the same intersection presheaf $\mathbb{X}$, then
the choice of a retraction $\overline{r}$ induces a left-adjoint
(resp. right-adjoint) for the forgetful functor $F$ from $B_{\alpha,\beta}^{k}$-manifolds
to $C^{k}$-manifolds, which actually independs of $\overline{r}$.
In particular, if $B$ is fully absorbing, then $F$ is ambidextrous
adjoint.}
\begin{proof}
Immediate from the results and discussions above.
\end{proof}
\begin{cor}
With the same notations of the last theorem, if $B$ is fully left-absorbing
(resp. fully right-absorbing), then $\mathbf{Diff}_{\alpha,\beta}^{B,k}(\mathbb{X})$
has all small colimits (resp. small limits) that exist in $\mathbf{Diff}^{k}$.
If $B$ is fully absorbing, then the same applies for limits and colimits
simultaneously. In particular, in this last case $\mathbf{Diff}_{\alpha,\beta}^{B,k}(\mathbb{X})$
has finite products and coproducts.
\end{cor}
\begin{proof}
Just apply Theorem B together with the preservation of small colimits/limits
by left/right-adjoint functors \citep{maclane,handbook} and recall
that the category of $C^{k}$-manifolds has finite products and coproducts
\citep{models_infinitesimal_analysis ,smooth_spaces_baez}.
\end{proof}

\section*{Acknowledgments}

The first author was supported by CAPES (grant number 88887.187703/2018-00).

\end{document}